\documentclass{amsart}

\usepackage{amssymb,amsmath,stmaryrd,mathrsfs}
\def\definetac{\newif\iftac}    % Can't define a \newif inside another \if!
\ifx\tactrue\undefined
  \definetac
  \ifx\state\undefined\tacfalse\else\tactrue\fi
\fi
\iftac\else\usepackage{amsthm}\fi
\usepackage[all,2cell]{xy}
\UseAllTwocells
\usepackage{enumitem}
\usepackage{xcolor}
\definecolor{darkgreen}{rgb}{0,0.45,0} 
\usepackage[pagebackref,colorlinks,citecolor=darkgreen,linkcolor=darkgreen]{hyperref}
\usepackage{mathtools}          % for all sorts of things
\usepackage{tikz}
\usepackage{braket}             % for \Set, etc.

\usepackage{url}                % for citations to web sites
\usepackage{xspace}             % put spaces after a \command in text

\makeatletter
\let\ea\expandafter

\def\mdef#1#2{\ea\ea\ea\gdef\ea\ea\noexpand#1\ea{\ea\ensuremath\ea{#2}\xspace}}
\def\alwaysmath#1{\ea\ea\ea\global\ea\ea\ea\let\ea\ea\csname your@#1\endcsname\csname #1\endcsname
  \ea\def\csname #1\endcsname{\ensuremath{\csname your@#1\endcsname}\xspace}}

\DeclareRobustCommand\widecheck[1]{{\mathpalette\@widecheck{#1}}}
\def\@widecheck#1#2{%
    \setbox\z@\hbox{\m@th$#1#2$}%
    \setbox\tw@\hbox{\m@th$#1%
       \widehat{%
          \vrule\@width\z@\@height\ht\z@
          \vrule\@height\z@\@width\wd\z@}$}%
    \dp\tw@-\ht\z@
    \@tempdima\ht\z@ \advance\@tempdima2\ht\tw@ \divide\@tempdima\thr@@
    \setbox\tw@\hbox{%
       \raise\@tempdima\hbox{\scalebox{1}[-1]{\lower\@tempdima\box
\tw@}}}%
    {\ooalign{\box\tw@ \cr \box\z@}}}

\newcount\foreachcount

\def\foreachletter#1#2#3{\foreachcount=#1
  \ea\loop\ea\ea\ea#3\@alph\foreachcount
  \advance\foreachcount by 1
  \ifnum\foreachcount<#2\repeat}

\def\foreachLetter#1#2#3{\foreachcount=#1
  \ea\loop\ea\ea\ea#3\@Alph\foreachcount
  \advance\foreachcount by 1
  \ifnum\foreachcount<#2\repeat}

\def\definescr#1{\ea\gdef\csname s#1\endcsname{\ensuremath{\mathscr{#1}}\xspace}}
\foreachLetter{1}{27}{\definescr}
\def\definecal#1{\ea\gdef\csname c#1\endcsname{\ensuremath{\mathcal{#1}}\xspace}}
\foreachLetter{1}{27}{\definecal}
\def\definebold#1{\ea\gdef\csname b#1\endcsname{\ensuremath{\mathbf{#1}}\xspace}}
\foreachLetter{1}{27}{\definebold}
\def\definebb#1{\ea\gdef\csname l#1\endcsname{\ensuremath{\mathbb{#1}}\xspace}}
\foreachLetter{1}{27}{\definebb}
\def\definefrak#1{\ea\gdef\csname f#1\endcsname{\ensuremath{\mathfrak{#1}}\xspace}}
\foreachletter{1}{9}{\definefrak} % \fi is something else!
\foreachletter{10}{27}{\definefrak}
\def\definebar#1{\ea\gdef\csname #1bar\endcsname{\ensuremath{\overline{#1}}\xspace}}
\foreachLetter{1}{27}{\definebar}
\foreachletter{1}{8}{\definebar} % \hbar is something else!
\foreachletter{9}{15}{\definebar} % \obar is something else!
\foreachletter{16}{27}{\definebar}
\def\definetil#1{\ea\gdef\csname #1til\endcsname{\ensuremath{\widetilde{#1}}\xspace}}
\foreachLetter{1}{27}{\definetil}
\foreachletter{1}{27}{\definetil}
\def\definehat#1{\ea\gdef\csname #1hat\endcsname{\ensuremath{\widehat{#1}}\xspace}}
\foreachLetter{1}{27}{\definehat}
\foreachletter{1}{27}{\definehat}
\def\definechk#1{\ea\gdef\csname #1chk\endcsname{\ensuremath{\widecheck{#1}}\xspace}}
\foreachLetter{1}{27}{\definechk}
\foreachletter{1}{27}{\definechk}
\def\defineul#1{\ea\gdef\csname u#1\endcsname{\ensuremath{\underline{#1}}\xspace}}
\foreachLetter{1}{27}{\defineul}
\foreachletter{1}{27}{\defineul}

\def\autofmt@n#1\autofmt@end{\mathrm{#1}}
\def\autofmt@b#1\autofmt@end{\mathbf{#1}}
\def\autofmt@l#1#2\autofmt@end{\mathbb{#1}\mathsf{#2}}
\def\autofmt@c#1#2\autofmt@end{\mathcal{#1}\mathit{#2}}
\def\autofmt@s#1#2\autofmt@end{\mathscr{#1}\mathit{#2}}
\def\autofmt@f#1\autofmt@end{\mathsf{#1}}
\def\autofmt@u#1\autofmt@end{\underline{\smash{\mathsf{#1}}}}
\def\autofmt@U#1\autofmt@end{\underline{\underline{\smash{\mathsf{#1}}}}}
\def\autofmt@h#1\autofmt@end{\widehat{#1}}
\def\autofmt@r#1\autofmt@end{\overline{#1}}
\def\autofmt@t#1\autofmt@end{\widetilde{#1}}
\def\autofmt@k#1\autofmt@end{\check{#1}}

\def\auto@drop#1{}
\def\autodef#1{\ea\ea\ea\@autodef\ea\ea\ea#1\ea\auto@drop\string#1\autodef@end}
\def\@autodef#1#2#3\autodef@end{%
  \ea\def\ea#1\ea{\ea\ensuremath\ea{\csname autofmt@#2\endcsname#3\autofmt@end}\xspace}}
\def\autodefs@end{blarg!}
\def\autodefs#1{\@autodefs#1\autodefs@end}
\def\@autodefs#1{\ifx#1\autodefs@end%
  \def\autodefs@next{}%
  \else%
  \def\autodefs@next{\autodef#1\@autodefs}%
  \fi\autodefs@next}

\DeclareSymbolFont{bbold}{U}{bbold}{m}{n}
\DeclareSymbolFontAlphabet{\mathbbb}{bbold}

\newcommand{\bbone}{\ensuremath{\mathbbb{1}}\xspace}

\mdef\delbar{\overline{\partial}}

\newcommand{\inv}{^{-1}}

\mdef\hf{\textstyle\frac12 }
\mdef\thrd{\textstyle\frac13 }
\mdef\qtr{\textstyle\frac14 }

\newcommand{\op}{^{\mathrm{op}}}

\SelectTips{cm}{}
\newdir{ >}{{}*!/-10pt/@{>}}    % extra spacing for tail arrows in XYpic
\newcommand{\pushoutcorner}[1][dr]{\save*!/#1+1.2pc/#1:(1,-1)@^{|-}\restore}

\mdef\Id{\mathrm{Id}}
\mdef\id{\mathrm{id}}
\alwaysmath{ell}
\alwaysmath{infty}
\alwaysmath{odot}
\def\frc#1/#2.{\frac{#1}{#2}}   % \frc x^2+1 / x^2-1 .
\mdef\ten{\mathrel{\otimes}}

\mdef\sqten{\mathrel{\boxtimes}}

\DeclareRobustCommand\widecheck[1]{{\mathpalette\@widecheck{#1}}}
\def\@widecheck#1#2{%
    \setbox\z@\hbox{\m@th$#1#2$}%
    \setbox\tw@\hbox{\m@th$#1%
       \widehat{%
          \vrule\@width\z@\@height\ht\z@
          \vrule\@height\z@\@width\wd\z@}$}%
    \dp\tw@-\ht\z@
    \@tempdima\ht\z@ \advance\@tempdima2\ht\tw@ \divide\@tempdima\thr@@
    \setbox\tw@\hbox{%
       \raise\@tempdima\hbox{\scalebox{1}[-1]{\lower\@tempdima\box
\tw@}}}%
    {\ooalign{\box\tw@ \cr \box\z@}}}

\DeclareMathOperator\colim{colim}

\DeclareMathOperator\Ho{Ho}

\newcommand{\ot}{\ensuremath{\leftarrow}}

\mdef\we{\overset{\sim}{\longrightarrow}}
\mdef\leftwe{\overset{\sim}{\longleftarrow}}

\let\xto\xrightarrow

\def\rightarrowtailfill@{\arrowfill@{\Yright\joinrel\relbar}\relbar\rightarrow}
\newcommand\xrightarrowtail[2][]{\ext@arrow 0055{\rightarrowtailfill@}{#1}{#2}}

\def\twoheadrightarrowfill@{\arrowfill@{\relbar\joinrel\relbar}\relbar\twoheadrightarrow}
\newcommand\xtwoheadrightarrow[2][]{\ext@arrow 0055{\twoheadrightarrowfill@}{#1}{#2}}

\def\slashedarrowfill@#1#2#3#4#5{%
  $\m@th\thickmuskip0mu\medmuskip\thickmuskip\thinmuskip\thickmuskip
   \relax#5#1\mkern-7mu%
   \cleaders\hbox{$#5\mkern-2mu#2\mkern-2mu$}\hfill
   \mathclap{#3}\mathclap{#2}%
   \cleaders\hbox{$#5\mkern-2mu#2\mkern-2mu$}\hfill
   \mkern-7mu#4$%
}
\def\rightslashedarrowfill@{%
  \slashedarrowfill@\relbar\relbar\mapstochar\rightarrow}
\newcommand\xslashedrightarrow[2][]{%
  \ext@arrow 0055{\rightslashedarrowfill@}{#1}{#2}}
\mdef\hto{\xslashedrightarrow{}}
\mdef\htoo{\xslashedrightarrow{\quad}}

\def\toiso{\xto{\smash{\raisebox{-.5mm}{$\scriptstyle\sim$}}}}

\long\def\my@drawfill#1#2;{%
\@skipfalse
\fill[#1,draw=none] #2;
\@skiptrue
\draw[#1,fill=none] #2;
}
\newif\if@skip
\newcommand{\skipit}[1]{\if@skip\else#1\fi}
\newcommand{\drawfill}[1][]{\my@drawfill{#1}}

\newif\ifhyperref
\@ifpackageloaded{hyperref}{\hyperreftrue}{\hyperreffalse}
\iftac
  \let\your@state\state
  \def\state#1{\gdef\currthmtype{#1}\your@state{#1}}
  \let\your@staterm\staterm
  \def\staterm#1{\gdef\currthmtype{#1}\your@staterm{#1}}
  \let\defthm\newtheorem
  \def\currthmtype{}
  \ifhyperref
    \def\autoref#1{\ref*{label@name@#1}~\ref{#1}}
  \else
    \def\autoref#1{\ref{label@name@#1}~\ref{#1}}
  \fi
  \AtBeginDocument{%
    \let\old@label\label%
    \def\label#1{%
      {\let\your@currentlabel\@currentlabel%
        \edef\@currentlabel{\currthmtype}%
        \old@label{label@name@#1}}%
      \old@label{#1}}
  }
\else
  \ifhyperref
    \def\defthm#1#2{%
      \newtheorem{#1}{#2}[section]%
      \expandafter\def\csname #1autorefname\endcsname{#2}%
      \expandafter\let\csname c@#1\endcsname\c@thm}
  \else
    \def\defthm#1#2{\newtheorem{#1}[thm]{#2}}
    \ifx\SK@label\undefined\let\SK@label\label\fi
    \let\old@label\label
    \let\your@thm\@thm
    \def\@thm#1#2#3{\gdef\currthmtype{#3}\your@thm{#1}{#2}{#3}}
    \def\currthmtype{}
    \def\label#1{{\let\your@currentlabel\@currentlabel\def\@currentlabel%
        {\currthmtype~\your@currentlabel}%
        \SK@label{#1@}}\old@label{#1}}
    \def\autoref#1{\ref{#1@}}
  \fi
\fi

\newtheorem{thm}{Theorem}[section]

\defthm{cor}{Corollary}
\defthm{prop}{Proposition}
\defthm{lem}{Lemma}
\defthm{sch}{Scholium}
\defthm{assume}{Assumption}
\defthm{claim}{Claim}
\defthm{conj}{Conjecture}
\defthm{hyp}{Hypothesis}
\defthm{fact}{Fact}
\iftac\theoremstyle{plain}\else\theoremstyle{definition}\fi
\defthm{defn}{Definition}
\defthm{notn}{Notation}
\iftac\theoremstyle{plain}\else\theoremstyle{remark}\fi
\defthm{rmk}{Remark}
\defthm{eg}{Example}
\defthm{egs}{Examples}
\defthm{ex}{Exercise}
\defthm{ceg}{Counterexample}
\defthm{warn}{Warning}
\defthm{con}{Construction}

\def\thmqedhere{\expandafter\csname\csname @currenvir\endcsname @qed\endcsname}

\setitemize[1]{leftmargin=2em}
\setenumerate[1]{leftmargin=*}

\iftac
  \let\c@equation\c@subsection
\else
  \let\c@equation\c@thm
\fi
\numberwithin{equation}{section}

\@ifpackageloaded{mathtools}{\mathtoolsset{showonlyrefs,showmanualtags}}{}

\alwaysmath{alpha}
\alwaysmath{beta}
\alwaysmath{gamma}
\alwaysmath{Gamma}
\alwaysmath{delta}
\alwaysmath{Delta}
\alwaysmath{epsilon}
\mdef\ep{\varepsilon}
\alwaysmath{zeta}
\alwaysmath{eta}
\alwaysmath{theta}
\alwaysmath{Theta}
\alwaysmath{iota}
\alwaysmath{kappa}
\alwaysmath{lambda}
\alwaysmath{Lambda}
\alwaysmath{mu}
\alwaysmath{nu}
\alwaysmath{xi}
\alwaysmath{pi}
\alwaysmath{rho}
\alwaysmath{sigma}
\alwaysmath{Sigma}
\alwaysmath{tau}
\alwaysmath{upsilon}
\alwaysmath{Upsilon}
\alwaysmath{phi}
\alwaysmath{Pi}
\alwaysmath{Phi}
\mdef\ph{\varphi}
\alwaysmath{chi}
\alwaysmath{psi}
\alwaysmath{Psi}
\alwaysmath{omega}
\alwaysmath{Omega}

\makeatother

\usetikzlibrary{calc}	% allowing calculations
\usetikzlibrary{shapes} % clouds are what we want here...

\newcommand{\TikzMorphismStyle}[4]{
 \draw #2 node {$\cdot$};
 \draw #3 node {$\cdot$};
 \draw[#1] ($#2!#4!#3$) -- ($#2!1.0-#4!#3$);
}

\newcommand{\TikzMorphism}[4]{
 \TikzMorphismStyle{white,double,ultra thick}{#2}{#3}{#4}
 \TikzMorphismStyle{#1}{#2}{#3}{#4}
}

\newcommand{\TikzMorphismStyleWithControls}[5]{
 \draw #2 node {$\cdot$};
 \draw #3 node {$\cdot$};
 \draw[#1] ($#2!#5!#3$) .. controls #4 .. ($#2!1.0-#5!#3$);
}

\newcommand{\TikzMorphismWithControls}[5]{
 \TikzMorphismStyleWithControls{white,double,ultra thick}{#2}{#3}{#4}{#5}
 \TikzMorphismStyleWithControls{#1}{#2}{#3}{#4}{#5}
}

\newcommand{\TikzBiprodCube}[4]{
  \TikzMorphism{#1}{(#2,#3)}{(#2+#4,#3)}{0.1}
  \TikzMorphism{#1}{(#2+#4,#3)}{(#2+2*#4,#3)}{0.1}

  \TikzMorphism{#1}{(#2,#3-#4)}{(#2 + #4,#3-#4)}{0.1}
  \TikzMorphism{#1}{(#2+#4,#3-#4)}{(#2+2*#4,#3-#4)}{0.1}

  \TikzMorphism{#1}{(#2,#3-2*#4)}{(#2 + #4,#3-2*#4)}{0.1}
  \TikzMorphism{#1}{(#2+#4,#3-2*#4)}{(#2+2*#4,#3-2*#4)}{0.1}

  \TikzMorphism{#1}{(#2,#3)}{(#2,#3-#4)}{0.1}
  \TikzMorphism{#1}{(#2,#3-#4)}{(#2,#3-2*#4)}{0.1}

  \TikzMorphism{#1}{(#2+#4,#3)}{(#2+#4,#3-#4)}{0.1}
  \TikzMorphism{#1}{(#2+#4,#3-#4)}{(#2+#4,#3-2*#4)}{0.1}

  \TikzMorphism{#1}{(#2+2*#4,#3)}{(#2+2*#4,#3-#4)}{0.1}
  \TikzMorphism{#1}{(#2+2*#4,#3-#4)}{(#2+2*#4,#3-2*#4)}{0.1}
  
  \TikzLabelCubeCenter{#2}{#3}{#4}
}

\newcommand{\TikzLabelSource}[3]{
  \node[anchor=south] at (#1-0.5*#3,#2-0.25*#3) {$q_0$};
}
\newcommand{\TikzLabelSink}[3]{
  \node[anchor=north] at (#1-0.5*#3,#2-1.75*#3) {$q'_0$};
}
\newcommand{\TikzLabelAdjacentVertices}[3]{
  \node[anchor=west]  at (#1+2*#3,#2-#3) {$q_1$};
  \node[anchor=north] at (#1+#3,#2-2*#3) {$q_2$};
}
\newcommand{\TikzLabelCubeCenter}[3]{
  \node[anchor=north west] at (#1+#3,#2-#3) {$b$};
}
\newcommand{\TikzLabelZeroInCofiber}[3]{
  \node[anchor=east] at (#1-1.7*#3,#2-#3) {$z$};
}

\newcommand{\TikzClouds}[3]{
  \node[cloud, cloud puffs=15, minimum width=1.3*#3cm, minimum height=#3cm, align=center, draw] (cloud) at (#1+2.4*#3,#2-#3) {};
  \node[cloud, cloud puffs=15, minimum width=#3cm, minimum height=1.3*#3cm, align=center, draw] (cloud) at (#1+#3,#2-2.4*#3) {};
  \TikzLabelAdjacentVertices{#1}{#2}{#3}
}

\newcommand{\TikzCubeNClouds}[4]{
  \TikzClouds{#2}{#3}{#4}
  \TikzBiprodCube{#1}{#2}{#3}{#4}
}

\newcommand{\TikzSourceForCube}[3]{
  \TikzMorphism{->}{(#1-0.5*#3,#2-0.25*#3)}{(#1+#3,#2-#3)}{0.1}
  \TikzLabelSource{#1}{#2}{#3}
}
\newcommand{\TikzSinkForCube}[3]{
  \TikzMorphism{<-}{(#1-0.5*#3,#2-1.75*#3)}{(#1+#3,#2-#3)}{0.1}
  \TikzLabelSink{#1}{#2}{#3}
}
\newcommand{\TikzSquareForCube}[3]{
  \TikzSourceForCube{#1}{#2}{#3}
  \TikzSinkForCube{#1}{#2}{#3}
  \TikzMorphism{->}{(#1-0.5*#3,#2-0.25*#3)}{(#1-1.7*#3,#2-#3)}{0.1}
  \TikzMorphism{<-}{(#1-0.5*#3,#2-1.75*#3)}{(#1-1.7*#3,#2-#3)}{0.1}
  \TikzLabelZeroInCofiber{#1}{#2}{#3}
}

\newcommand{\TikzCubeWithSource}[4]{
  \TikzCubeNClouds{#1}{#2}{#3}{#4}
  \TikzSourceForCube{#2}{#3}{#4}
}
\newcommand{\TikzCubeWithSink}[4]{
  \TikzCubeNClouds{#1}{#2}{#3}{#4}
  \TikzSinkForCube{#2}{#3}{#4}
}
\newcommand{\TikzCubeWithSquare}[4]{
  \TikzCubeNClouds{#1}{#2}{#3}{#4}
  \TikzSquareForCube{#2}{#3}{#4}
}

\newcommand{\TikzQuiverWithSource}[3]{
  \TikzClouds{#1}{#2}{#3}

  \TikzMorphismStyleWithControls{white,double,ultra thick}{(#1-0.5*#3,#2-0.25*#3)}{(#1+2*#3,#2-#3)}{(#1+0.5*#3,#2-0.75*#3)}{0.04}
  \TikzMorphismWithControls{->}{(#1-0.5*#3,#2-0.25*#3)}{(#1+#3,#2-2*#3)}{(#1+0.6*#3,#2-1.1*#3)}{0.04}
  \TikzMorphismStyleWithControls{->}{(#1-0.5*#3,#2-0.25*#3)}{(#1+2*#3,#2-#3)}{(#1+0.5*#3,#2-0.75*#3)}{0.04}  

  \TikzLabelSource{#1}{#2}{#3}
}
\newcommand{\TikzQuiverWithSink}[3]{
  \TikzClouds{#1}{#2}{#3}

  \TikzMorphismStyleWithControls{white,double,ultra thick}{(#1-0.5*#3,#2-1.75*#3)}{(#1+2*#3,#2-#3)}{(#1+0.5*#3,#2-1.6*#3)}{0.04}
  \TikzMorphismWithControls{<-}{(#1-0.5*#3,#2-1.75*#3)}{(#1+#3,#2-2*#3)}{(#1+0.3*#3,#2-1.8*#3)}{0.04}
  \TikzMorphismStyleWithControls{<-}{(#1-0.5*#3,#2-1.75*#3)}{(#1+2*#3,#2-#3)}{(#1+0.5*#3,#2-1.6*#3)}{0.04}

  \TikzLabelSink{#1}{#2}{#3}
}

\newcommand{\TikzReflectionStrategyLeft}{0.0}
\newcommand{\TikzReflectionStrategyMiddle}{3.2}
\newcommand{\TikzReflectionStrategyRight}{6.0}

\newcommand{\TikzReflectionStrategyRowOne}{13.85}
\newcommand{\TikzReflectionStrategyTransitionOne}{10.0}
\newcommand{\TikzReflectionStrategyRowTwo}{9.15}
\newcommand{\TikzReflectionStrategyTransitionTwo}{5.35}
\newcommand{\TikzReflectionStrategyRowThree}{4.5}
\newcommand{\TikzReflectionStrategyTransitionThree}{0.5}
\newcommand{\TikzReflectionStrategyRowFour}{0.0}

\newcommand{\TikzReflectionStrategy}{
  \begin{tikzpicture}
    \begin{scope}[scale=.85]
    
    \TikzQuiverWithSource{\TikzReflectionStrategyLeft}{\TikzReflectionStrategyRowOne}{1.0}
    \node[anchor=east] at (\TikzReflectionStrategyLeft-1.0,\TikzReflectionStrategyRowOne-1.3) {$Q\colon$};
    
    \draw[->,dashed] (\TikzReflectionStrategyLeft+1.0,\TikzReflectionStrategyTransitionOne+0.4) -- (\TikzReflectionStrategyLeft+1.0,\TikzReflectionStrategyTransitionOne-0.4);
    \node[anchor=east] at (\TikzReflectionStrategyLeft+0.8,\TikzReflectionStrategyTransitionOne) {add a biproduct cube};
    
    \TikzCubeWithSource{->}{\TikzReflectionStrategyLeft}{\TikzReflectionStrategyRowTwo}{1.0}
    \node[anchor=east] at (\TikzReflectionStrategyLeft-1.0,\TikzReflectionStrategyRowTwo-1.0) {$Q_1\colon$};

    \draw[->,dashed] (\TikzReflectionStrategyLeft+1.0,\TikzReflectionStrategyTransitionTwo+0.4) -- (\TikzReflectionStrategyLeft+1.0,\TikzReflectionStrategyTransitionTwo-0.4);
    \node[anchor=east] at (\TikzReflectionStrategyLeft+0.8,\TikzReflectionStrategyTransitionTwo) {make the cube invertible};

    \TikzCubeWithSource{<->}{\TikzReflectionStrategyLeft}{\TikzReflectionStrategyRowThree}{1.0}
    \node[anchor=east] at (\TikzReflectionStrategyLeft-1.0,\TikzReflectionStrategyRowThree-1.0) {$Q_2\colon$};

    \draw[->,dashed] (\TikzReflectionStrategyLeft+1.0,\TikzReflectionStrategyTransitionThree+0.4) .. controls (\TikzReflectionStrategyLeft+1.0,\TikzReflectionStrategyTransitionThree-0.1) .. (\TikzReflectionStrategyLeft+1.4,\TikzReflectionStrategyTransitionThree-0.4);
    \node[anchor=east] at (\TikzReflectionStrategyLeft+0.8,\TikzReflectionStrategyTransitionThree) {construct a cofiber};

    \TikzQuiverWithSink{\TikzReflectionStrategyRight}{\TikzReflectionStrategyRowOne}{1.0}
    \node[anchor=west] at (\TikzReflectionStrategyRight+3.7,\TikzReflectionStrategyRowOne-1.3) {$\colon Q'$};

    \draw[->,dashed] (\TikzReflectionStrategyRight+1.0,\TikzReflectionStrategyTransitionOne+0.4) -- (\TikzReflectionStrategyRight+1.0,\TikzReflectionStrategyTransitionOne-0.4);
    \node[anchor=west] at (\TikzReflectionStrategyRight+1.2,\TikzReflectionStrategyTransitionOne) {add a biproduct cube};

    \TikzCubeWithSink{<-}{\TikzReflectionStrategyRight}{\TikzReflectionStrategyRowTwo}{1.0}
    \node[anchor=west] at (\TikzReflectionStrategyRight+3.7,\TikzReflectionStrategyRowTwo-1.0) {$\colon Q'_1$};

    \draw[->,dashed] (\TikzReflectionStrategyRight+1.0,\TikzReflectionStrategyTransitionTwo+0.4) -- (\TikzReflectionStrategyRight+1.0,\TikzReflectionStrategyTransitionTwo-0.4);
    \node[anchor=west] at (\TikzReflectionStrategyRight+1.2,\TikzReflectionStrategyTransitionTwo) {make the cube invertible};

    \TikzCubeWithSink{<->}{\TikzReflectionStrategyRight}{\TikzReflectionStrategyRowThree}{1.0}
    \node[anchor=west] at (\TikzReflectionStrategyRight+3.7,\TikzReflectionStrategyRowThree-1.0) {$\colon Q'_2$};
    
    \draw[->,dashed] (\TikzReflectionStrategyRight+1.0,\TikzReflectionStrategyTransitionThree+0.4) .. controls (\TikzReflectionStrategyRight+1.0,\TikzReflectionStrategyTransitionThree-0.1) .. (\TikzReflectionStrategyRight+0.6,\TikzReflectionStrategyTransitionThree-0.4);
    \node[anchor=west] at (\TikzReflectionStrategyRight+1.2,\TikzReflectionStrategyTransitionThree) {construct a fiber};

    \TikzCubeWithSquare{<->}{\TikzReflectionStrategyMiddle}{\TikzReflectionStrategyRowFour}{1.0}  
    
    \end{scope}
  \end{tikzpicture}
}

\tikzset{lab/.style={auto,font=\scriptsize}} % arrow labels
\usetikzlibrary{arrows}

\newcommand{\D}{\sD}
\newcommand{\E}{\sE}

\autodefs{\cDER\cCat\cGrp\cCAT\cMONCAT\bSet\cProf\bCat
\cPro\cMONDER\cBICAT\ncomp\nid\ncoker\niso\nIm\sProf\ncyl\fC\fZ\fT\nhocolim\nholim\cPsNat\ncoll\bsSet\cAlg\cIm\cI\cP}
\let\sset\bsSet
\def\cPDER{\ensuremath{\mathcal{PD}\mathit{ER}}\xspace}

\def\ho{\mathscr{H}\!\mathit{o}\xspace}

\let\oldboxtimes\boxtimes
\def\boxtimes{\mathrel{\oldboxtimes}}

\newcommand{\fib}{\mathsf{fib}}
\newcommand{\cof}{\mathsf{cof}}

\newcommand{\sse}{\stackrel{\mathrm{s}}{\sim}}

\def\ccsub{_{\mathrm{cc}}}
\def\pdh(#1,#2){\llbracket #1,#2\rrbracket}
\def\ldh(#1,#2){\llbracket #1,#2\rrbracket\ccsub}
\def\pend(#1){\pdh(#1,#1)}
\def\lend(#1){\ldh(#1,#1)}

\def\shift#1#2{{#1}^{#2}}

\def\DTl#1#2#3#4#5#6#7{%
  \xymatrix@C=3pc{{#1} \ar[r]^-{#2} &
    {#3} \ar[r]^-{#4} &
    {#5} \ar[r]^-{#6} &
    {#7}
  }}

\newsavebox{\tvabox}
\savebox\tvabox{\hspace{1mm}\begin{tikzpicture}[>=latex',baseline={(0,-.18)}]
  \draw[->] (0,.1) -- +(1,0);
  \node at (.5,0) {$\scriptscriptstyle\bot$};
  \draw[->] (1,-.1) -- +(-1,0);
  \draw[->] (1,-.2) -- +(-1,0);
\end{tikzpicture}\hspace{1mm}}

\newcommand{\Ch}{\mathrm{Ch}}
\newcommand{\Mod}{\mathrm{Mod}}

\title{Tilting theory for trees via stable homotopy theory}
\author{Moritz Groth}
\address{MPIM, Vivatsgasse 7, 53111 Bonn, Germany}
\email{mgroth@mpim-bonn.mpg.de}

\author{Jan \v{S}\v{t}ov\'{\i}\v{c}ek}
\address{Department of Algebra, Charles University in Prague, Sokolovska 83, 186 75 Praha~8, Czech Republic}
\email{stovicek@karlin.mff.cuni.cz}

\date{\today}

\thanks{%
The first named author was supported by the Dutch Science Foundation (NWO). The second named author was supported by grant GA\v{C}R P201/12/G028 from the Czech Science Foundation.%
}

\begin{document}

\begin{abstract}
We show that variants of the classical reflection functors from quiver representation theory exist in any abstract stable homotopy theory, making them available for example over arbitrary ground rings, for quasi-coherent modules on schemes, in the differential-graded context, in stable homotopy theory as well as in the equivariant, motivic, and parametrized variant thereof. As an application of these equivalences we obtain abstract tilting results for trees valid in all these situations, hence generalizing a result of Happel. 

The main tools introduced for the construction of these reflection functors are homotopical epimorphisms of small categories and one-point extensions of small categories, both of which are inspired by similar concepts in homological algebra.
\end{abstract}

\maketitle

\tableofcontents

\section{Introduction}
\label{sec:intro}

In \cite{gabriel:unzerlegbar} Gabriel classified (up to Morita equivalence) connected hereditary representation-finite algebras over an algebraically closed field~$k$ by means of their associated quivers. More precisely, an algebra is of the above kind if and only if it is the path algebra $kQ$, where the graph of the quiver $Q$ belongs to a certain list of explicit graphs (which is completely exhausted by the Dynkin diagrams of type $ADE$). 

Later Bern{\v{s}}te{\u\i}n, Gel$'$fand, and Ponomarev \cite{BGP:reflection} gave an elegant proof of this same result based on \emph{reflection functors}. Given a quiver~$Q$ and a vertex $q\in Q$ which is a source (no arrow ends at $q$) or a sink (no arrow starts at $q$), the reflection $Q'=\sigma_qQ$ of $Q$ is the quiver obtained by reversing the orientations of all edges adjacent to~$q$. Associated to this reflection at the level of quivers, there is a reflection functor $\Mod(kQ)\to\Mod(kQ')$ between the abelian categories of representations of the respective path algebras $kQ$ and $kQ'$. These reflection functors are not equivalences, but the deviation from this is well-understood (see \S\ref{sec:reflectrep} for a precise definition of these functors).

Happel \cite{happel:fdalgebra} showed that if the quiver contains no \emph{oriented} cycles, then the total derived functors of these reflection functors induce exact equivalences 
\[
D(kQ)\stackrel{\Delta}{\simeq} D(kQ')
\]
between the derived categories of the path-algebras (and a similar equivalence $D(\cA^Q) \simeq D(\cA^{Q'})$ for an arbitrary abelian category $\cA$ was later established by Ladkani~\cite{Ladkani07}). \emph{The main aim of this paper is to show that for oriented trees this result is a formal consequence of stability alone and it hence has variants in many other contexts arising in algebra, geometry, and topology}. Let us make this more precise and then illustrate the added generality by mentioning some examples.

Recall that there are many different approaches to axiomatic homotopy theory including model categories, $\infty$-categories, and derivators. Here we choose to work with \emph{derivators}, a notion introduced by Heller \cite{heller:htpy}, Grothendieck \cite{grothendieck:derivateurs}, Franke \cite{franke:adams}, and others. A derivator is some kind of a minimal extension of a classical derived category or homotopy category to a framework with a powerful calculus of homotopy (co)limits and also homotopy Kan extensions. In this extension, homotopy (co)limit constructions are characterized by ordinary universal properties, making them accessible to elementary methods from category theory. A derivator is \emph{stable} if it admits a zero object and if homotopy pushouts and homotopy pullbacks coincide. Stable derivators provide an enhancement of triangulated categories (see \cite{groth:ptstab}). For example associated to a ring~$R$ there is the stable derivator~$\D_R$ of unbounded chain complexes over~$R$, enhancing the classical derived category $D(R)$.

An important construction at the level of derivators is given by shifting: given a derivator \D and a small category $B$, there is the derivator $\D^B$ of coherent $B$-shaped diagrams in \D, which is stable as soon as \D is. Since a quiver $Q$ is simply a graph and hence has an associated free category, this shifting operation can be applied to quivers. For example, if we consider the stable derivator $\D_k$ of a field~$k$, then the shifted derivator~$\D_k^Q$ is equivalent to the stable derivator $\D_{kQ}$ associated to the path-algebra~$kQ$.

The main aim of this paper is then to show that if $Q$ is an oriented tree and if $Q'$ is obtained from $Q$ by an arbitrary reorientation, then \emph{for every stable derivator}~\D there is a pseudo-natural equivalence of derivators
\[
\D^Q\simeq\D^{Q'}.
\]
In the terminology of \cite{gst:basic} we thus show that $Q$ and $Q'$ are \emph{strongly stably equivalent} (\autoref{thm:reflection} and its corollaries). Choosing specific stable derivators, this gives us refined variants of the above-mentioned result of Happel. If we take the stable derivator $\D_k$ of a field, then we obtain equivalences of derivators 
\[
\D_{kQ}\simeq \D_k^Q\simeq \D_k^{Q'}\simeq \D_{kQ'},
\]
and, in particular, exact equivalences $D(kQ)\stackrel{\Delta}{\simeq} D(kQ')$ of the underlying triangulated categories. 

However, the same result is also true for the derivator $\D_R$ of a ring~$R$, for the derivator $\D_X$ of a (quasi-compact and quasi-separated) scheme~$X$, for the derivator $\D_A$ of a differential-graded algebra~$A$, for the derivator $\D_E$ of a (symmetric) ring spectrum~$E$, and for other stable derivators arising for example in stable homotopy theory as well as in its equivariant, motivic, or parametrized variants. Moreover, these equivalences are pseudo-natural with respect to exact morphisms and hence commute, in particular, with various restriction of scalar functors, induction and coinduction functors, and localizations and colocalizations (see \cite[\S5]{gst:basic} for more examples of stable derivators and many references).

By a combinatorial argument in order to obtain arbitrary reorientations of trees it is enough to construct reflection functors at sources and sinks inducing such pseudo-natural equivalences. We mimic the classical construction from algebra, however we have to adapt certain steps significantly to make them work in this more general context. Let $q_0\in Q$ be a source in an oriented tree. The classical reflection functors are roughly obtained by taking the sum of all outgoing morphisms at the source, passing to the cokernel of this map, and then using the structure maps of the biproduct in order to obtain a representation of the reflected quiver (see \S\ref{sec:reflectrep} for more details).

The main reason why we have to work harder in this more general context is that the final, innocent looking step cannot be performed that easily with abstract coherent diagrams: in such a diagram, we cannot simply replace the projections of a biproduct by the corresponding inclusions. Instead this is achieved by passing to larger diagrams which encode the biproduct objects and all the necessary projection and injection maps simultaneously. 

A first step towards this is obtained by encoding finite biproducts by means of certain $n$-cubes of length two, the biproduct object sitting in the center (see \S\ref{sec:biproducts}). For two summands the picture to have in mind is 
\[
\xymatrix{
0\ar[r]\ar[d]&X\ar[r]\ar[d]&0\ar[d]\\
Y\ar[r]\ar[d]&X\oplus Y\ar[r]\ar[d]&Y\ar[d]\\
0\ar[r]&X\ar[r]&0.
}
\]
Using the basic theory of strongly bicartesian $n$-cubes established in \cite[\S8]{gst:basic}, this can be generalized to the case of finitely many arguments -- the problem still being that we cannot simply pass from the projections to the injections.

The easy but key observation to solve this is that in such biproduct diagrams all `length two morphisms' are invertible. This suggests that these diagrams should arise as restrictions of similar diagrams where the shape is given by some kind of `invertible $n$-cube of length two' coming with all the necessary structure maps in both directions. In order to make this precise, we introduce the concept of \emph{homotopical epimorphism} (see \S\ref{sec:epi}), and verify that this indeed works if we consider the standalone cube (see \S\ref{sec:epibiprod}).

The next step is then to inductively take into account the parts of the representations which lie outside of this $n$-cube. This is achieved by means of \emph{one-point extensions} (see \S\ref{sec:onepoint}). Setting up things in the correct way, we end up with a category which contains both the original quiver and the reflected one as subcategories, such that we understand which representations of this larger category come from the respective quivers. This shows finally that the two quivers are strongly stably equivalent (see \S\ref{sec:reflection}).

This paper is a sequel to \cite{gst:basic} which in turn builds on \cite{groth:ptstab} and \cite{gps:mayer}, and as a such is part of a \emph{formal study of stability} (similarly to \cite{gps:additivity} which is a \emph{formal study of the interaction of stability and monoidal structure}). The aim of these papers and its sequels is to develop a formal, stable calculus which is available in arbitrary abstract stable homotopy theories, including typical situations arising in algebra, geometry, and topology (like the ones mentioned above). We expect this calculus to be rather rich, and we will develop further aspects of it somewhere else. For example, in \cite{gst:Dynkin-A} we study in more detail abstract representation theory of Dynkin quivers of type~$A$. As applications this yields universal tilting modules for $A_n$-quivers, i.e., certain spectral bimodules realizing important functors (e.g., reflection functors, Coxeter functors) in arbitrary stable homotopy theories. Besides conceptual explanations of fractionally Calabi--Yau dimensions this also allows for applications to the various notions of higher triangulations. And in \cite{gst:acyclic}, we establish abstract reflection functors and universal tilting modules for more general shapes, including arbitrary \emph{acyclic quivers}.

The content of the sections is as follows. In \S\S\ref{sec:derivators}--\ref{sec:stable} we recall some basics on derivators and, in particular, stable derivators. In \S\ref{sec:biproducts} we describe how to model finite biproducts in stable derivators by means of $n$-cubes. In \S\ref{sec:reflectrep} we recall the classical construction of reflection functors in representation theory of quivers, and describe the strategy on how to generalize this to abstract stable derivators. In~\S\ref{sec:epi} we introduce and study homotopical epimorphisms. Specializing a combinatorial detection criterion for homotopy exact squares from \cite{gps:mayer} we also obtain such a criterion for homotopical epimorphisms which is crucial in later sections. In \S\ref{sec:epibiprod} we establish a key example of a homotopical epimorphism, allowing us to describe finite biproducts in stable derivators by `invertible $n$-cubes'. In \S\ref{sec:onepoint} we introduce one-point extensions, and show that homotopical epimorphisms are stable under one-point extensions, in a way that we can control the essential images of the associated restriction functors. In \S\ref{sec:reflection} we assemble the above results to construct the reflection functors, and deduce that arbitrary reorientations of oriented trees yield strongly stably equivalent quivers.

\textit{Acknowledgments.} The authors thank an anonymous referee for helpful comments on an earlier version of this paper. This includes simplified proofs of \autoref{prop:oneexact} and \autoref{thm:onepointessim}.

\section{Review of derivators}
\label{sec:derivators}

In this section and the following one we include a short review of derivators and stable derivators, mainly to fix some notation and to quote a few results which are of constant use in later sections. More details can be found in  \cite{groth:ptstab}, its sequel \cite{gps:mayer}, and the many references therein. 

Let \cCat denote the 2-category of small categories, \cCAT the 2-category of not necessarily small categories. The category with one object and its identity morphism only is denoted by~$\bbone$. Note that objects $a\in A$ correspond to functors $a\colon\bbone\to A$ under the natural isomorphism $A\cong A^\bbone$.

A \textbf{prederivator} is simply a 2-functor $\D\colon\cCat\op\to\cCAT.$\footnote{Recall that $\cCat\op$ is obtained from $\cCat$ by reversing the orientation of the functors but not of the natural transformations. This reflects the fact that, following Heller~\cite{heller:htpy} and Franke~\cite{franke:adams}, our convention for derivators is based on \emph{diagrams}. There is an alternative, but isomorphic approach based on \emph{presheaves}, i.e., contravariant functors, in which case also the orientations of the natural transformations should be changed; see for example \cite{grothendieck:derivateurs,cisinski:idcm}.} Morphisms of prederivators are pseudo-natural transformations and transformations of prederivators are modifications so that we obtain a 2-category $\cPDER$ of prederivators (see \cite{borceux1}). The category $\D(A)$ is the category of \textbf{coherent $A$-shaped diagrams} in \D. If $u\colon A\to B$ is a functor, then we denote the \textbf{restriction functor} by $u^*\colon\D(B) \to \D(A)$. 

In the case of a functor $a\colon\bbone\to A$ classifying an object, $a^\ast\colon\D(A)\to\D(\bbone)$ is an \textbf{evaluation functor}, taking values in the \textbf{underlying category} $\D(\bbone)$. If $f\colon X\to Y$ is a morphism in $\D(A)$, then its image is denoted by $f_a\colon X_a\to Y_a$. These evaluation functors allow us to assign to any coherent diagram $X\in\D(A)$ an \textbf{underlying (incoherent) diagram} $A\to\D(\bbone)$. The resulting functor $\D(A)\to\D(\bbone)^A$ however is, in general, far from being an equivalence, and coherent diagrams are hence not determined by their underlying diagrams, even not up to isomorphism. Nevertheless, frequently we draw coherent diagrams as usual and say that such a diagram has the form of or looks like its underlying diagram. 

Without additional axioms we cannot perform any constructions. A \emph{derivator} is a prederivator which `allows for a well-behaved calculus of Kan extensions', satisfying key properties of the calculus available in model categories, $\infty$-categories, other approaches to higher category theory, as well as in ordinary categories. This seemingly abstract concept turns out to capture many typical constructions in homological algebra and homotopy theory, including
\begin{enumerate}
\item suspensions, loops, (co)fibers, Baratt--Puppe sequences \cite{groth:ptstab},
\item homotopy orbits and homotopy fixed points of actions of discrete groups, leading to (co)homology of groups in the algebraic context,
\item homotopy (co)ends and homotopy tensor products of functors \cite{gps:additivity}, 
\item spectrifications of prespectrum objects,
\end{enumerate}
and many more including the reflection functors as we show in this paper. The axiomatization of this calculus is as follows.

If a restriction functor $u^\ast\colon\D(B)\to\D(A)$ admits a left adjoint $u_!\colon \D(A)\to\D(B)$, then we refer to it as a \textbf{left Kan extension functor}. A right adjoint $u_\ast\colon\D(A)\to\D(B)$ is a \textbf{right Kan extension functor}. In the examples of interest these are really \emph{homotopy Kan extension functors}. We nevertheless follow the established terminology for~$\infty$-categories and refer to the functors as \emph{Kan extensions}, and there is no risk of confusion as `categorical' Kan extensions are meaningless in the context of an abstract prederivator. In the special case that $B=\bbone$ is the terminal category and we hence consider the unique functor $\pi=\pi_A\colon A\to \bbone$, the functor $\pi_!=\mathrm{colim}_A$ is a \textbf{colimit functor} and $\pi_\ast=\mathrm{lim}_A$ a \textbf{limit functor}. 

To actually work with these Kan extensions, one encodes the pointwise formulas which are known to be satisfied in the examples of interest (see \cite[X.3.1]{maclane} for the classical case). To make this axiom precise, we consider the special cases of comma squares
\begin{equation}
\vcenter{
\xymatrix{
(u/b)\ar[r]^-p\ar[d]_-\pi\drtwocell\omit{}&A\ar[d]^-u&&(b/u)\ar[r]^-q\ar[d]_-\pi&A\ar[d]^-u\\
\bbone\ar[r]_-b&B,&&\bbone\ar[r]_-b&B\ultwocell\omit{},
}
}
\label{eq:Der4}
\end{equation}
which come with canonical transformations $u\circ p\to b\circ\pi$ and $b\circ\pi\to u\circ q.$

\begin{defn}
  A \textbf{derivator} is a prederivator $\D\colon\cCat\op\to\cCAT$ with the following properties.
  \begin{itemize}[leftmargin=4em]
  \item[(Der1)] $\D\colon \cCat\op\to\cCAT$ takes coproducts to products.  In particular, $\D(\emptyset)$ is the terminal category.
  \item[(Der2)] For any $A\in\cCat$, a morphism $f\colon X\to Y$ is an isomorphism in $\D(A)$ if and only if the morphisms $f_a\colon X_a\to Y_a, a\in A,$ are isomorphisms in $\D(\bbone).$
  \item[(Der3)] Each functor $u^*\colon \D(B) \to\D(A)$ has both a left adjoint $u_!$ and a right adjoint $u_*$.
  \item[(Der4)] For any functor $u\colon A\to B$ and any $b\in B$ the canonical transformation
\[ 
\pi_! p^* \to \pi_! p^* u^* u_! \to \pi_! \pi^* b^* u_! \to b^* u_!
\]
is an isomorphism as is the canonical transformation
\[
b^* u_* \to \pi_* \pi^* b^* u_* \to \pi_* q^* u^* u_* \to \pi_* q^*.
\]
  \end{itemize}
\end{defn}

The axioms (Der1) and (Der3) together imply that $\D(A)$ has small categorical coproducts and products, hence, in particular, initial objects and final objects. These are the only actual 1-categorical (co)limits which must exist in each derivator. We will add a few explanatory comments on (Der4) after the following list of examples.

\begin{egs}\label{egs:prederivators}
~ % This is to fix the refernces to the examples
\begin{enumerate}
\item Any category \bC gives rise to a \textbf{represented} prederivator~$y(\bC)$ defined by $y(\bC)(A) \coloneqq \bC^A.$ Its underlying category is equivalent to \bC itself. This prederivator is a derivator if and only if \bC is complete and cocomplete, in which case the functors $u_!,u_\ast$ are ordinary Kan extension functors. 
\item A \emph{Quillen model category} \bC (see e.g.~\cite{quillen:ha,hovey:modelcats}) with \emph{weak equivalences} \bW has an underlying \textbf{homotopy derivator} $\ho(\bC)$ defined by formally inverting the pointwise weak equivalences $\ho(\bC)(A) \coloneqq (\bC^A)[(\bW^A)^{-1}]$. The underlying category of $\ho(\bC)$ is the homotopy category $\bC[\bW^{-1}]$ of \bC, and the functors $u_!,u_*$ are derived versions of the functors of $y(\bC)$ (see~\cite{cisinski:idcm} for the general case and \cite{groth:ptstab} for an easy proof in the case of combinatorial model categories). \item Similarly, let $\cC$ be an $\infty$-category in the sense of Joyal \cite{joyal:barca} and Lurie \cite{lurie:HTT} (see \cite{groth:scinfinity} for an introduction). Then we obtain a prederivator $\ho(\cC)$ by sending $A$ to the homotopy category of $\cC^{N(A)}$ where $N(A)$ denotes the nerve of $A$. For a sketch proof that for a complete and cocomplete $\infty$-category this yields a derivator, the \textbf{homotopy derivator} of $\cC$, we refer to~\cite{gps:mayer}.  
\item Given a derivator \D and $B\in\cCat$ we can define the \textbf{shifted derivator} $\shift\D B$ by $\shift\D B(A) \coloneqq \D(B\times A)$ (see~\cite[Theorem~1.25]{groth:ptstab}). Moreover, the \textbf{opposite derivator} $\D\op$ is defined by $\D\op(A) \coloneqq \D(A\op)\op$. These two constructions satisfy many compatibilities with the previous three examples. Also shifting and the passage to opposites are compatible in the sense that we have $(\shift\D B)\op \cong \shift{(\D\op)}{B\op}$.
\end{enumerate}
\end{egs}

For more specific examples of derivators we refer the reader to \cite[Examples~5.4]{gst:basic}. The existence of opposites of derivators implies that the duality principle applies to the theory of derivators. The shifting operation is of central importance in this paper. Given a derivator \D and $B\in\cCat$, then $\D^B$ is the homotopy theory of coherent $B$-shaped diagrams in \D. We will later use this in the special case where $B=Q$ is a quiver.

With these examples in mind, we see that (Der4) expresses the idea that Kan extensions in derivators are pointwise. In the notation of the axiom and of \eqref{eq:Der4}, let us consider the derivator $\D=y(\bC)$ represented by a complete and cocomplete category~$\bC$. Given a diagram $X\colon A\to \bC$, the canonical maps in (Der4) are the isomorphisms
\[
\colim_{(u/b)} X\circ p \stackrel{\cong}{\to} \mathrm{LKan}_u(X)_b\qquad\text{and}\qquad
\mathrm{RKan}_u(X)_b\stackrel{\cong}{\to}\lim_{(b/u)} X\circ q,
\]
which allow us classically to compute Kan extensions in terms of colimits and limits.

A \textbf{morphism} of derivators is simply a morphism of underlying prederivators, i.e., a pseudo-natural transformation $F\colon \D\to\E$. Similarly, given 
two such morphisms $F,G\colon\D\to\E$, a \textbf{natural transformation} $F\to G$ is a modification. Thus, we define the $2$-category $\cDER$ of derivators as a full sub-2-category of $\cPDER$.

We now recall the notion of a \emph{homotopy exact square} of small categories, which arguably is the main tool in the study of derivators as it allows us to extend many key facts about classical Kan extensions to the context of an abstract derivator and hence to \autoref{egs:prederivators} and in particular to \cite[Examples~5.4]{gst:basic}. Given a derivator \D and a natural transformation living in a square
\begin{equation}
  \vcenter{\xymatrix{
      D\ar[r]^p\ar[d]_q \drtwocell\omit{\alpha} &
      A\ar[d]^u\\
      B\ar[r]_v &
      C
    }}\label{eq:hoexactsq'}
\end{equation}
of small categories, we obtain \textbf{canonical mate-transformations}
\begin{gather}
  q_! p^* \stackrel{\eta}{\to} q_! p^* u^* u_! \xto{\alpha^*} q_! q^* v^* v_! \stackrel{\epsilon}{\to} v^* u_!  \mathrlap{\qquad\text{and}}\label{eq:hoexmate1'}\\
  u^* v_* \stackrel{\eta}{\to} p_* p^* u^* v_* \xto{\alpha^*} p_* q^* v^* v_* \stackrel{\epsilon}{\to} p_* q^*\label{eq:hoexmate2'}
\end{gather}
Here, $\eta$ denotes the adjunction units of the respective adjunctions and $\epsilon$ the respective adjunction counits. It can be shown that \eqref{eq:hoexmate1'} is an isomorphism if and only if \eqref{eq:hoexmate2'} is one. 

A square \eqref{eq:hoexactsq'} is by definition \textbf{homotopy exact} if the canonical mates \eqref{eq:hoexmate1'} and \eqref{eq:hoexmate2'} are isomorphisms. By (Der4) the comma squares \eqref{eq:Der4} are homotopy exact, and many further examples can be established (see for example \cite{ayoub:1,ayoub:2,maltsiniotis:exact} and \cite{groth:ptstab,gps:mayer,gst:basic}.) Here, we only collect the examples needed in this paper, and we refer to the above references for more details.

\begin{egs}\label{egs:htpy}
~ % This is to fix the refernces to the examples
\begin{enumerate}
\item \emph{Kan extensions along fully faithful functors are again fully faithful.} If $u\colon A\to B$ is fully faithful, then the square $u\circ \id=u\circ \id$ is homotopy exact. Thus, the unit $\eta\colon\id\to u^\ast u_!$ and the counit $\epsilon\colon u^\ast u_\ast\to\id$ are isomorphisms  (\cite[Proposition~1.20]{groth:ptstab}).
\item \emph{Kan extensions and restrictions in unrelated variables commute.} For functors $u\colon A\to B$ and $v\colon C\to D$ the squares
\begin{equation}
\vcenter{
\xymatrix{
A\times C\ar[r]^-{u\times\id}\ar[d]_-{\id\times v}&B\times C\ar[d]^-{\id\times v}\\
A\times D\ar[r]_-{u\times\id}&B\times D
}}
\end{equation}
are homotopy exact (\cite[Proposition~2.5]{groth:ptstab}).
\item \emph{Right adjoint functors are \textbf{homotopy final}.} If $u\colon A\to B$ is a right adjoint, then the square
\[
\xymatrix{
A\ar[r]^-u\ar[d]_-{\pi_A}&B\ar[d]^-{\pi_B}\\
\bbone\ar[r]_-\id&\bbone
}
\]
is homotopy exact, i.e., the canonical mate $\mathrm{colim}_Au^\ast\to\mathrm{colim}_B$ is an isomorphism (\cite[Proposition~1.18]{groth:ptstab}). 
\item \emph{Homotopy exact squares are compatible with pasting.} The passage to mate-transformation is functorial with respect to horizontal and vertical pasting. Consequently, horizontal and vertical pastings of homotopy exact squares are homotopy exact (\cite[Lemma~1.14]{groth:ptstab}). 
\end{enumerate}
\end{egs}

Since Kan extensions and restrictions in unrelated variables commute, there are \emph{parametrized versions} of restrictions and Kan extensions. More precisely, associated to a derivator~\D and a functor $u\colon A\to B$ there are adjunctions of derivators
\begin{equation}
(u_!,u^\ast)\colon\D^A\rightleftarrows\D^B\qquad\text{and}\qquad
(u^\ast,u_\ast)\colon\D^B\rightleftarrows\D^A.\label{eq:Kanadjunction}
\end{equation}
(See \cite[\S2]{groth:ptstab} for details on adjunctions of derivators.) In particular, if $u$ is fully faithful, then the morphisms $u_!,u_\ast$ induce equivalences of derivators onto the respective essential images, a result we use without further reference in later sections.

Finally, we use the following notation. Given a (pre)derivator \D, then we write $X\in\D$ in order to indicate that $X\in\D(A)$ for some $A\in\cCat$.

\section{Stable derivators and strongly stably equivalent categories}
\label{sec:stable}

In this short section we recall some basics about pointed and stable derivators, mostly to fix some notation (again, for more details see \cite{groth:ptstab,gps:mayer}). Examples of such derivators arise from pointed or stable model categories and similarly for complete and cocomplete $\infty$-categories. Thus stable derivators describe aspects of the calculus of homotopy Kan extensions available in such examples arising in algebra, geometry, and topology. We will also recall the notion of strongly stably equivalent categories  introduced in \cite{gst:basic} .

As shown in \cite{groth:ptstab}, the following definition is equivalent to the original one of Maltsiniotis \cite{m:k-theory-deriv}.

\begin{defn} 
A derivator \D is \textbf{pointed} if $\D(\bbone)$ has a zero object. 
\end{defn}

If $\D$ is pointed then so are $\D^B$ and $\D\op$. It follows that the categories $\D(A)$ have zero objects which are preserved by restriction and Kan extension functors. 

In the pointed context, Kan extensions along inclusions of \emph{cosieves} and \emph{sieves} `extend diagrams by zero objects'. Recall that $u\colon A\to B$ is a \textbf{sieve} if it is fully faithful, and for any morphism $b\to u(a)$ in $B$, there exists an $a'\in A$ with $u(a')=b$. Dually, there is the notion of a \textbf{cosieve}, and we know already that Kan extensions along (co)sieves are fully faithful (see \autoref{egs:htpy}).

\begin{lem}[{\cite[Prop.~1.23]{groth:ptstab}}]\label{lem:extbyzero}
Let \D be a pointed derivator and let $u\colon A\to B$ be a sieve. Then $u_\ast\colon\D^A\to\D^B$ induces an equivalence onto the full subderivator of $\D^B$ spanned by all diagrams $X\in\D^B$ such that $X_b$ is zero for all $b\notin u(A)$.
\end{lem}
\noindent
The functor $u_\ast$ is \textbf{right extension by zero}. Dually, left Kan extensions along cosieves give \textbf{left extension by zero}.

In the framework of pointed derivators one can define \textbf{suspensions} and \textbf{loops}, \textbf{cofibers} and \textbf{fibers}, and similar constructions. In particular, we have adjunctions of derivators
\[
(\Sigma,\Omega)\colon\D\rightleftarrows\D\quad\text{and}\quad
(\cof,\fib)\colon\D^{[1]}\rightleftarrows\D^{[1]}
\]
where $[1]$ is the category $(0\to 1)$ with two objects and a unique non-identity morphism.
Let us sketch the construction of the suspension $\Sigma$ and the cofiber $\cof$, the other two functors being dual. The commutative square $\square=[1]^2=[1]\times[1],$
\begin{equation}
\vcenter{
\xymatrix@-.5pc{
    (0,0)\ar[r]\ar[d] &
    (1,0)\ar[d]\\
    (0,1)\ar[r] &
    (1,1),
  }
}
\label{eq:square}
\end{equation}
has full subcategories $i_\ulcorner\colon\ulcorner\to\square$ and $i_\lrcorner\colon\lrcorner\to\square$ obtained by removing $(1,1)$ and $(0,0)$, respectively. Since both inclusions are fully faithful, so are the associated Kan extension functors (see \autoref{egs:htpy}). A square $Q\in\D(\square)$ is \textbf{cocartesian} if it lies in the essential image of $(i_\ulcorner)_!$. Dually, the \textbf{cartesian} squares are precisely the ones in the essential image of $(i_\lrcorner)_\ast$. Given a coherent morphism $(f\colon X\to Y)\in\D([1])$, we obtain a cocartesian square in \D of the form
\[
\xymatrix{
X\ar[r]^-f\ar[r]\ar[d]&Y\ar[d]^-{\cof(f)}\\
0\ar[r]&Z
}
\]
first by right extending by zero and then by left Kan extension. A restriction of this cocartesian square along the obvious functor $[1]\to\square$ defines the cofiber of~$f$. In the special case where $(f\colon X\to 0)$ itself is already obtained by right extension by zero, the cofiber object is called the suspension~$\Sigma X$. Thus there is a defining cocartesian square looking like
\[
\xymatrix{
X\ar[r]\ar[r]\ar[d]&0\ar[d]\\
0\ar[r]&\Sigma X.
}
\]
A \textbf{(coherent) cofiber sequence} is a coherent diagram
\[
\xymatrix{
X\ar[r]\ar[d]&Y\ar[r]\ar[d]&0\ar[d]\\
0\ar[r]&Z\ar[r]&W
}
\]
such that both squares are cocartesian and the two corners vanish as indicated. An arbitrary coherent morphism $(X\to Y)\in\D([1])$ can be extended to a cofiber sequence by a right extension by zero followed by a left Kan extension. (Note that this defines a cofiber sequence functor $\D^{[1]}\to\D^{[2]\times[1]}$ which is an equivalence onto its essential image.) Since the compound square is cocartesian (\cite[Corollary~4.10]{gps:mayer}), the object $W$ is canonically isomorphic to $\Sigma X$. In particular, using this isomorphism we can associate an underlying \textbf{\emph{incoherent} cofiber sequence}
\[
X\to Y\to Z\to \Sigma X\quad \text{in}\;\D(\bbone)
\]
to any coherent cofiber sequence.

\begin{defn}
A pointed derivator is \textbf{stable} if the classes of cocartesian squares and cartesian squares coincide.
\end{defn}
\noindent

The homotopy derivators of stable model categories or stable $\infty$-categories are stable. If \D is stable then so is the shifted derivator $\D^B$ (\cite[Proposition~4.3]{groth:ptstab}) as is $\D\op$. Recall that a derivator is \textbf{strong} if it satisfies:
  \begin{itemize}[leftmargin=4em]
  \item[(Der5)] For any $A$, the induced functor $\D(A\times [1]) \to \D(A)^{[1]}$ is full and essentially surjective.
  \end{itemize}
This property does not play an essential role in the basic \emph{theory} of derivators. It is mainly used if one wants to relate \emph{properties} of stable derivators to the existence of \emph{structure} on its values. The precise form of this axiom depends on the reference: in \cite{heller:htpy} a stronger version is used. Represented derivators and homotopy derivators associated to model categories or $\infty$-categories are strong as are shifts and opposites of strong derivators. 

\begin{thm}[{\cite[Theorem~4.16]{groth:ptstab}}]\label{thm:triang}
Let \D be a strong, stable derivator. Then the categories $\D(A),A\in\cCat,$ admit (canonical) triangulations. 
\end{thm}

These triangulations are canonical in the following sense. Recall that a morphism between stable derivators is exact if it preserves zero object, cartesian and cocartesian squares. For more details on this notion see for example \cite{groth:ptstab}.

\begin{prop}[{\cite[Proposition~4.18]{groth:ptstab}}]\label{thm:triangcan}
Let $F\colon\D\to\E$ be an exact morphism of strong, stable derivators. Then the components $F_A\colon\D(A)\to\E(A)$ can be endowed with the structure of an exact functor with respect to the triangulations of \autoref{thm:triang}.
\end{prop}

Thus, like stable model categories and stable $\infty$-categories, stable derivators provide an enhancement of triangulated categories. 

\begin{notn}
We denote by $\cDER_{\mathrm{St},\mathrm{ex}}\subseteq\cDER$ the $2$-category of stable derivators, exact morphisms, and arbitrary transformations. This is clearly a $2$-category since identity morphisms are exact as are compositions of exact morphisms.
\end{notn}

Recall that the shifting operation defines a $2$-functor
\[
\cCat\op\times\cDER\to\cDER\colon (A,\D)\mapsto \D^A.
\]
In particular, for every $A\in\cCat$ we can restrict the $2$-functor $(-)^A\colon\cDER\to\cDER$ to obtain
\[
(-)^A\colon\cDER_{\mathrm{St},\mathrm{ex}}\to\cDER.
\]
With the idea that stable derivators enhance triangulated categories, the following notion was introduced in \cite{gst:basic}.

\begin{defn}\label{defn:sse}
Two small categories $A$ and $A'$ are \textbf{strongly stably equivalent}, in notation $A\sse A',$ if there is a pseudo-natural equivalence 
\[
\Phi\colon(-)^A\simeq(-)^{A'}\colon\cDER_{\mathrm{St},\mathrm{ex}}\to\cDER.
\]
Such a pseudo-natural equivalence is a \textbf{strong stable equivalence}.
\end{defn}

To be more specific, such a strong stable equivalence $\Phi$ consists of
\begin{enumerate}
\item equivalences of derivators $\Phi_\D\colon\D^A\simeq\D^{A'}$, \D a stable derivator, and
\item for every exact morphism $F\colon\D\to\E$ of stable derivators of an invertible modification $\gamma_F$,
\[
\xymatrix{
\D^A\ar[r]^-{\Phi_\D}\ar[d]_-{F^A}\drtwocell\omit{\cong}&\D^{A'}\ar[d]^-{F^{A'}}\\
\E^A\ar[r]_-{\Phi_\E}&\E^{A'},
}
\]
such that some coherence properties are satisfied.
\end{enumerate}

More generally, in this paper we establish many examples of pseudo-natural equivalences of suitable $2$-functors $T,T'\colon\cDER_{\mathrm{St},\mathrm{ex}}\to\cDER.$ When obvious from context, often we only specify the behavior of these $2$-functors on objects. Moreover, instead of always saying that there is a pseudo-natural equivalence $T\simeq T'$, we simply mention that there is a pseudo-natural equivalence $T(\D)\simeq T'(\D)$ for $\D\in\cDER_{\mathrm{St},\mathrm{ex}}$. In particular, often we do not make explicit the pseudo-naturality isomorphisms.

For every field $k$ there is the stable derivator $\D_k$ associated to the projective model structure on the category of unbounded chain complexes over~$k$. Recall that a module over the path-algebra $kQ$ of a quiver $Q$ is equivalently specified by a functor $Q\to\Mod(k)$, where, strictly speaking, $Q$ is the free category generated by the graph $Q$. This identification extends to chain complexes and implies that $\D_k^Q$ is equivalent to the derivator $\D_{kQ}$ of the path-algebra. In particular, for two strongly stably equivalent quivers $Q$ and~$Q'$ there are exact equivalences $D(kQ)\simeq D(kQ')$ of the underlying derived categories of the path-algebras, showing that the quivers are derived equivalent. 

However, the statement that two categories or quivers are \emph{strongly stably equivalent} is a much stronger statement as such equivalences then also have to exist for the derivator $\D_R$ of a ring~$R$, the derivator $\D_X$ of a (quasi-compact and quasi-separated) scheme~$X$, the derivator $\D_A$ of a differential-graded algebra~$A$, for the derivator $\D_E$ of a (symmetric) ring spectrum~$E$, and other examples arising for example in equivariant stable, motivic stable, or parametrized stable homotopy theory. Moreover, since strong stable equivalences are pseudo-natural with respect to exact morphisms all the above equivalences commute with various restriction and (co)induction of scalar functors as well as Bousfield (co)localizations (see \cite[\S5]{gst:basic} for many examples of stable derivators and references as well as further comments on \autoref{defn:sse}).

In order to show that certain categories or quivers are strongly stably equivalent we use suitable combinations of Kan extension morphisms as in \eqref{eq:Kanadjunction}. For that purpose it is convenient to recall the following lemma which allows us to `detect' (co)cartesian squares in larger coherent diagrams. 

\begin{lem}[{\cite[Prop.~3.10]{groth:ptstab}}]\label{lem:detection}
  Suppose $u\colon C\to B$ and $v\colon \Box \to B$ are functors, with $v$ injective on objects, and let $b = v(1,1)\in B$.
  Suppose furthermore that $b\notin u(C)$, and that the functor $\mathord\ulcorner \to (B\setminus b)/b$ induced by $v$ has a left adjoint.
  Then for any derivator \D and any $X\in \D(C)$, the square $v^* u_! X$ is cocartesian.
\end{lem}

We will also need its generalization to arbitrary \emph{cocones}. Given a small category $A$ let $A^\rhd$ be the \textbf{cocone} on~$A$, i.e., the category which is obtained from $A$ by adjoining a new terminal object $\infty$. The cocone construction $(-)^\rhd$ is obviously functorial in~$A$, and there is a natural fully faithful inclusion $i=i_A\colon A\to A^\rhd$. A cocone $X\in\D^{A^\rhd}$ is \textbf{colimiting} if it lies in the essential image of the left Kan extension functor $(i_A)_!\colon\D^A\to\D^{A^\rhd}$.

\begin{lem}[{\cite[Lemma~4.5]{gps:mayer}}]\label{lem:detectionplus}
Let $A\in\cCat$, and let $u\colon C\to B$, $v\colon A^\rhd\to B$ be functors.
Suppose that there is a full subcategory $B'\subseteq B$ such that
\begin{enumerate}
\item $u(C) \subseteq B'$ and $v(\infty) \notin B'$;
\item $vi\colon A\to B$ factors through the inclusion $B'\subseteq B$; and
\item the functor $A \to B' / v(\infty)$ induced by $v$ has a left adjoint.
\end{enumerate}
Then for any derivator \D and any $X\in\D(C)$, the diagram $v^* u_! X$ is in the essential image of $i_!$.
In particular, $(v^* u_! X)_\infty$ is the colimit of $i^* v^* u_! X$.
\end{lem}

\section{Finite biproducts via $n$-cubes}
\label{sec:biproducts}

For the construction of abstract reflection functors in \S\ref{sec:reflection} we will need a description of finite biproducts in stable derivators by means of certain coherent diagrams. For this purpose, in this section we briefly recall from \cite[\S4.1]{groth:ptstab} the construction of the biproduct of pairs of objects in stable derivators. Then, using basic results on strongly bicartesian $n$-cubes from \cite[Section~8]{gst:basic}, we extend this description to the case of finitely many summands.

Let \D be a stable derivator and let $[2]$ be the category $(0<1<2)$. For the case of two summands, we consider the category $[2]\times[2],$
\[
\xymatrix{
(0,0)\ar[r]\ar[d]&(1,0)\ar[r]\ar[d]&(2,0)\ar[d]\\
(0,1)\ar[r]\ar[d]&(1,1)\ar[r]\ar[d]&(2,1)\ar[d]\\
(0,2)\ar[r]&(1,2)\ar[r]&(2,2),
}
\]
and the following \emph{full} subcategories of $[2]\times[2]$ which we define by listing their objects,
\begin{enumerate}
\item $A_1\colon$ (1,2)-(2,1),
\item $A_2\colon$ (1,2)-(2,2)-(2,1), and
\item $A_3\colon$ (0,2)-(1,2)-(2,2)-(2,1)-(2,0).
\end{enumerate}
There are obvious inclusions $j_1\colon A_1\to A_2,$ $j_2\colon A_2\to A_3,$ and $j_3\colon A_3\to [2]\times[2],$ and the following sequence of Kan extensions
\begin{equation}
\D^{A_1}\stackrel{(j_1)_\ast}{\to}\D^{A_2}\stackrel{(j_2)_!}{\to}\D^{A_3}\stackrel{(j_3)_\ast}{\to}\D^{[2]\times[2]}\label{eq:biprod}
\end{equation}
consists of fully faithful morphisms of derivators (\autoref{egs:htpy}). Moreover, since $j_1$ is the inclusion of a sieve it follows from \autoref{lem:extbyzero} that $(j_1)_\ast$ is right extension by zero. Similarly, $j_2$ is the inclusion of a cosieve and hence $(j_2)_!$ is left extension by zero. Finally, a repeated application of \autoref{lem:detection} implies that $(j_3)_\ast$ amounts to forming four bicartesian squares. We summarize this construction in the following lemma (see \cite[\S4.1]{groth:ptstab} for more details).

\begin{lem}\label{lem:biprod}
Let \D be a stable derivator and let $\D^{[2]\times[2],\mathrm{ex}}\subseteq\D^{[2]\times[2]}$ be the full subderivator spanned by the coherent diagrams vanishing on the corners and making all squares bicartesian. Then \eqref{eq:biprod} induces a pseudo-natural equivalence $\D^{\bbone\sqcup\bbone}\simeq\D^{[2]\times[2],\mathrm{ex}}$.
\end{lem}

We only have to verify that a diagram in the image of \eqref{eq:biprod} also vanishes at the corner $(0,0)$. Note that (Der1) implies that there is a canonical equivalence $\D\times\D\simeq\D^{\bbone\sqcup\bbone}$. By the above, \eqref{eq:biprod} sends $(X,Y)\in\D\times\D\simeq\D^{\bbone\sqcup\bbone}=\D^{A_1}$ to a coherent diagram in $\D^{[2]\times[2]}$ looking like
\begin{equation}
\vcenter{
\xymatrix{
Z\ar[r]\ar[d]&\tilde X\ar[r]\ar[d]&0\ar[d]\\
\tilde Y\ar[r]\ar[d]&B\ar[r]\ar[d]&Y\ar[d]\\
0\ar[r]&X\ar[r]&0.
}
}
\label{eq:biproduct}
\end{equation}
\noindent
Since all squares in this diagram are bicartesian, since composites of bicartesian squares are again bicartesian (\cite[Corollary~4.10]{gps:mayer}) and since pullbacks of isomorphisms are again isomorphisms (\cite[Proposition~3.12]{groth:ptstab}), we obtain isomorphisms
\[
X\cong \tilde X,\qquad Y\cong \tilde Y, \qquad \text{and}\qquad Z\cong 0.
\]
Moreover, as a pullback over a zero object is a product and dually (\cite[Corollary~4.11]{gps:mayer}), it follows that $B$ is a biproduct of $X$ and $Y$. 

The diagram \eqref{eq:biproduct} actually is a coherent version of the biproduct object together with the corresponding inclusions and projections. Thus, \autoref{lem:biprod} can be read as saying that for a stable derivator $\D$, the derivator $\D^{\bbone\sqcup\bbone}\simeq \D\times\D$ is pseudo-naturally equivalent to the \emph{derivator of biproduct diagrams}. We want to emphasize that in \eqref{eq:biproduct} all `length two morphisms' 
\[
(i,0)\to(i,2)\quad\text{and}\quad (0,j)\to(2,j)
\]
are sent to isomorphisms. 

We now discuss a variant of this construction for biproducts of finitely many summands. For this purpose we recall from \cite[\S8]{gst:basic} some basic notation and terminology related to $n$-cubes in derivators. We denote the $n$-cube by $[1]^n=[1]\times\ldots\times[1]$. Note that there is an order-preserving isomorphism between $[1]^n$ and the power set of $\{1,\ldots,n\}$, and we allow ourselves to pass back and forth between these two descriptions without explicit mention. For $0\leq k\leq n$ we denote by $i_{\geq k}\colon[1]^n_{\geq k}\to[1]^n$ the inclusion of the full subcategory spanned by all subsets of cardinality at least $k$; see \autoref{fig:cube-inclusions}. There are similar full subcategories $[1]^n_{=k},[1]^n_{\leq k},[1]^n_{>k},$ and so on. In particular, the full subcategory $[1]^n_{=n-1}\subseteq[1]^n$ is the \emph{discrete} category $n\cdot\bbone=\bbone\sqcup\ldots\sqcup\bbone$ with $n$ objects. 

\begin{figure}
%\centering{
\xymatrix@-1.5pc{
{\phantom{000}} && {\phantom{100}} &&&
{\phantom{000}} && 100 \ar[dr] \ar'[d][dd]&&&
000 \ar[rr] \ar[dr] \ar[dd] && 100 \ar[dr] \ar'[d][dd] \\
& {\phantom{010}} && 110 \ar[dd]&&
& 010 \ar[rr] \ar[dd] && 110 \ar[dd]&&
& 010 \ar[rr] \ar[dd] && 110 \ar[dd]\\
{\phantom{001}} && 101 \ar[dr]&&&
001 \ar[dr] \ar'[r][rr] && 101 \ar[dr]&&&
001 \ar[dr] \ar'[r][rr] && 101 \ar[dr]\\
& 011 \ar[rr] && 111,&&
& 011 \ar[rr] && 111,&&
& 011 \ar[rr] && 111.
}
%}
\caption{The inclusions $[1]^3_{\geq 2}\subseteq[1]^3_{\geq 1}\subseteq[1]^3_{\geq 0}=[1]^3.$}
\label{fig:cube-inclusions}
\end{figure}
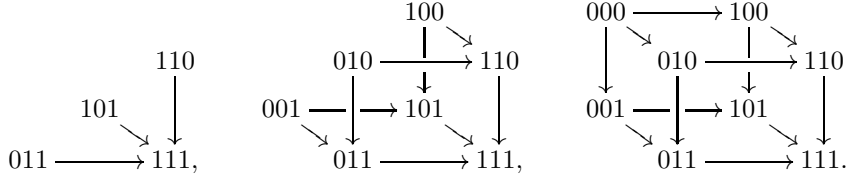

\begin{defn}
Let \D be a derivator. An $n$-cube $X\in\D^{[1]^n}$ is \textbf{strongly cartesian} if it lies in the essential image of $(i_{\geq n-1})_\ast\colon\D^{[1]_{\geq n-1}^n}\to\D^{[1]^n}$. An $n$-cube $X$ is \textbf{cartesian} if it lies in the essential image of $(i_{\geq 1})_\ast$.
\end{defn}

Since $i_{\geq k}$ is fully faithful, the same is true for $(i_{\geq k})_\ast$, and an $n$-cube $X$ lies in the essential image of $(i_{\geq k})_\ast$ if and only if the unit $\eta\colon X\to (i_{\geq k})_\ast(i_{\geq k})^\ast(X)$ is an isomorphism. In particular, $X$ is strongly cartesian if and only if the unit $X\to(i_{\geq n-1})_\ast(i_{\geq n-1})^\ast(X)$ is an isomorphism. 

Inspired by the work of Goodwillie \cite{goodwillie:II}, in \cite[\S8]{gst:basic} we have shown that an $n$-cube in a derivator is strongly cartesian if and only if all \emph{subcubes} are cartesian (see \cite[Theorem~8.4]{gst:basic}) if and only if all \emph{subsquares} are cartesian (see \cite[Corollary~8.12]{gst:basic}). This implies the equivalence of the third and the fourth characterization in the following theorem.

\begin{thm}[{\cite[Theorem~7.1]{gps:mayer},\cite[Corollary~8.13]{gst:basic}}]
The following are equivalent for a pointed derivator \D.
\begin{enumerate}
\item The adjunction $(\Sigma,\Omega)\colon\D(\bbone)\to\D(\bbone)$ is an equivalence.
\item The adjunction $(\cof,\fib)\colon\D([1])\to\D([1])$ is an equivalence.
\item The derivator \D is stable.
\item An $n$-cube in \D, $n\geq 2,$ is strongly cartesian if and only if it is strongly cocartesian. 
\end{enumerate}
\end{thm}

An $n$-cube which is simultaneously strongly cartesian and strongly cocartesian is \textbf{strongly bicartesian}. In the case of $n=2$ this reduces to the classical notion of a \textbf{bicartesian square}. Similar to bicartesian squares also strongly bicartesian $n$-cubes enjoy a 2-out-of-3 property with respect to composition and cancellation (see \cite[\S8]{gst:basic} for the case of $n$-cubes).

With the construction of coherent diagrams for finite biproducts in mind, let us consider the functors
\begin{equation}
\vcenter{
\xymatrix{
n\cdot\bbone=[1]^n_{=n-1}\ar[r]^-{w_1}&[1]^n_{\geq n-1}\ar[r]^-{w_2}&[1]^n\ar[r]^-{w_3}& I\ar[r]^-{w_4}& [2]^n.
}
}
\label{eq:biproductsshape}
\end{equation}
Here $w_1$ and $w_2=i_{\geq n-1}$ are the obvious fully faithful inclusions and $w_4w_3\colon[1]^n\to[2]^n$ is the inclusion as the $n$-cube 
\[
[(1,\ldots,1),(2,\ldots,2)]=[1,2]\times\ldots\times[1,2].
\] 
Finally, $I\subseteq [2]^n$ is the full subcategory spanned by $[(1,\ldots,1),(2,\ldots,2)]$ and the $n$ additional corners 
\begin{equation}
(0,2,2,\ldots,2,2),\quad (2,0,2,2,\ldots,2,2),\quad\ldots,\quad (2,2,\ldots,2,2,0),
\label{eq:zeros}
\end{equation}
while $w_3$ and $w_4$ are the fully faithful inclusions arising from the obvious factorization. Since all four functors are fully faithful the same is true for the Kan extension functors
\begin{equation}
\D^{n\cdot\bbone}=\D^{[1]^n_{=n-1}}\stackrel{(w_1)_\ast}{\to}\D^{[1]^n_{\geq n-1}}\stackrel{(w_2)_\ast}{\to}
\D^{[1]^n}\stackrel{(w_3)_!}{\to}\D^I\stackrel{(w_4)_\ast}{\to}\D^{[2]^n}.
\label{eq:biproducts}
\end{equation}
We show next that in the case of a stable derivator \D the essential image consists of the full subderivator $\D^{[2]^n,\mathrm{ex}}\subseteq\D^{[2]^n}$ spanned by the coherent diagrams such that
\begin{enumerate}
\item all subcubes are strongly bicartesian,
\item the values at all corners are trivial, and
\item the maps $(i_1,\ldots,i_{k-1},0,i_{k+1},\ldots,i_n)\to (i_1,\ldots,i_{k-1},2,i_{k+1},\ldots,i_n)$ are sent to isomorphisms for all $i_1,\ldots,i_{k-1},i_{k+1},\ldots,i_n$ and $k$.
\end{enumerate}

The proposition will justify the claim that $\D^{[2]^n,\mathrm{ex}}$ is a derivator. Note that property (iii) is a consequence of (i) and (ii) but we emphasize it here as this property will prove important later on.

\begin{prop}\label{prop:biproducts}
Let \D be a stable derivator. Then \eqref{eq:biproducts} consists of fully faithful functors and induces a pseudo-natural equivalence $\D^{n\cdot\bbone}\to\D^{[2]^n,\mathrm{ex}}$.
\end{prop}
\begin{proof}
Since $w_1\colon[1]_{= n-1}^n\to[1]_{\geq n-1}^n$ is the inclusion of a sieve, $(w_1)_\ast$ is right extension by zero (\autoref{lem:extbyzero}). It follows that $(w_1)_\ast\colon\D^{[1]_{=n-1}^n}\to\D^{[1]_{\geq n-1}^n}$ induces an equivalence on the full subderivator spanned by diagrams satisfying this vanishing condition. By definition of strongly cartesian $n$-cubes, the composition $(w_2)_\ast(w_1)_\ast\colon\D^{[1]_{=n-1}^n}\to\D^{[1]^n}$ induces an equivalence onto the full subderivator of strongly cartesian $n$-cubes which vanish on the final object. Note next that $w_3\colon[1]^n\to I$ is the inclusion of a cosieve. Hence $(w_3)_!\colon\D^{[1]^n}\to\D^I$ is left extension by zero and induces an equivalence onto the full subderivator of $\D^I$ spanned by the diagrams which vanish on the corners \eqref{eq:zeros} (again by \autoref{lem:extbyzero}). This equivalence restricts further to an equivalence of the respective subderivators spanned by diagrams vanishing on the final object and making the $n$-cube strongly cartesian. It remains to understand the effect of $(w_4)_\ast$. It follows from a repeated application of \autoref{lem:detectionplus} that $(w_4)_\ast\colon\D^I\to\D^{[2]^n}$ amounts to adding $2^n-1$ strongly bicartesian $n$-cubes. The verification of this is a straightforward higher-dimensional variant of the corresponding verifications in dimension two as carried out in \cite[\S4.1]{groth:ptstab}, and we leave the details to the reader.

As an upshot, the composition \eqref{eq:biproducts} induces an equivalence onto the full subderivator of $\D^{[2]^n}$ spanned by the diagrams making all $n$-cubes strongly bicartesian (use \cite[Corollary~8.13]{gst:basic}) and vanishing on the corners \eqref{eq:zeros} as well as on the final vertex. Note that this implies that all subsquares are then also bicartesian (\cite[Corollary~8.12]{gst:basic}). As isomorphisms in coherent diagrams are stable under pullbacks, it follows from the vanishing conditions of the diagrams that also condition (iii) in the definition of $\D^{[2]^n,\mathrm{ex}}$ is satisfied, and hence that these diagrams actually vanish on all corners. Finally, these equivalences assemble to a pseudo-natural equivalence since all the Kan extensions are preserved by exact morphisms of derivators (see also~\cite[\S7]{ps:linearity}). 
\end{proof}

Note that the objects of $\D^{[2]^n,\mathrm{ex}}$ are coherent diagrams for finite biproducts $X_1\oplus\ldots\oplus X_n$ which also encode all partial biproducts together with all the inclusion and projection maps. Thus, the proposition makes precise the expected result that in the stable case the derivator $\D^{n\cdot\bbone}\simeq\D\times\ldots\times\D$ is pseudo-naturally equivalent to the derivator $\D^{[2]^n,\mathrm{ex}}$ of $n$-fold biproduct diagrams. We will use a variant of this observation in our construction of abstract reflection functors in~\S\ref{sec:reflection}.

\section{Reflection functors in representation theory}
\label{sec:reflectrep}

The main aim of this paper is to construct reflection functors for arbitrary stable derivators, and to show that they give rise to strongly stably equivalent quivers. We begin by recalling some details about the classical construction at the level of abelian categories of representations. By a quiver $Q$ we formally mean a quadruple $(Q_0,Q_1,s,t)$ consisting of a set of vertices $Q_0$, a set of arrows $Q_1$, and two maps $s,t\colon Q_1\to Q_0$ assigning to each arrow its source and target vertex, respectively. An \textbf{oriented tree} is a connected quiver without unoriented cycles. In other words, an oriented tree is obtained from an ordinary tree by equipping each edge with an orientation. All our quivers in the sequel will be finite; that is $Q_0$ and $Q_1$ will be finite sets. 

Given a quiver~$Q$ and a vertex $q\in Q$, there is the \textbf{reflected quiver} $\sigma_qQ$ obtained from $Q$ by changing the orientations of all arrows adjacent to~$q$. A vertex~$q_0$ of a quiver is a \textbf{sink} or a \textbf{source} if all edges adjacent to it have~$q_0$ as their target or source, respectively. 

These reflections at sources and sinks and the associated reflection functors between the respective module categories have played an important role in representation theory, e.g.,~in the classification of quivers of finite representation type (see for example \cite{gabriel:unzerlegbar,BGP:reflection,happel:dynkin}). For more details about tilting theory itself we refer to \cite{tilting}.

The classical construction of the reflection functors in the case of a source $q_0$ is as follows. Let $Q'=\sigma_{q_0}Q$ be the reflected quiver and let $R$ be an arbitrary ground ring. If $M\colon Q\to\Mod(R)$ is a representation, then the \textbf{reflection functor} 
\[
s_{q_0}^-\colon\Mod(RQ)\to\Mod(RQ')
\]
sends $M$ to the following representation $M'=s_{q_0}^-(M)\colon Q'\to \Mod(R)$. On vertices we have $M'_{q'}=M_{q'}$ for all $q'\neq q_0$, and similarly $M'_f=M_f$ on all edges which are not adjacent to $q_0$. If $q_0\to q_i,1\leq i\leq n,$ is the set of morphisms adjacent to $q_0$ in $Q$, then $M'_{q_0}$ is defined as
\begin{equation}
M'_{q_0}=\ncoker\big(M_{q_0}\to\bigoplus_{i=1,\ldots,n}M_{q_i}\big).
\label{eq:reflection}
\end{equation}
The structure maps $M'_{q_j}\to M'_{q_0}$ in the reflected direction are given by 
\begin{equation}
\xymatrix{
M_{q_j}\ar[r]^-{\iota_j}&\bigoplus_{i=1,\ldots,n}M_{q_i}\to M'_{q_0},
}
\label{eq:clreflect}
\end{equation}
where the undecorated morphism is the canonical map to the cokernel. It is easy to use the universal property of the cokernel to define this reflection functor $s_{q_0}^-$ on morphisms of representations. The case of a sink is dual. In particular, since $q_0$ is a sink in $Q'$, we obtain a functor
\[
s_{q_0}^+\colon\Mod(RQ')\to\Mod(RQ)
\]
which is easily checked to be a right adjoint to $s_{q_0}^-$.

Thus the situation is roughly summarized by the following diagram in which for the sake of simplicity we suppressed any additional branches which could potentially be attached to the vertices $M_{q_i},i\geq 1:$

\begin{equation}
\vcenter{
\xymatrix{
&&&M_{q_1}& \\
&M_{q_0}\ar@{-->}@/_0.8pc/[dl]\ar[r]&M_{q_1}\oplus M_{q_2}\oplus M_{q_3}\ar@/^0.8pc/[ur]^-{\pi_1}\ar@/_0.8pc/[dr]_-{\pi_3}\ar[rr]^-{\pi_2}\ar@{-->}@/_0.8pc/[ld]&&M_{q_2}\\
0\ar@{-->}[r]&M'_{q_0}&&M_{q_3}&
}}\label{eq:reflectionold}
\end{equation}

The adjunction $(s_{q_0}^-,s_{q_0}^+)\colon\Mod(RQ)\rightleftarrows\Mod(RQ')$ is \emph{not} an equivalence of abelian categories of representations (for more detail see e.g.~\cite{ASS:representation}). However, the following result was established by Happel. 

\begin{thm}[{\cite{happel:fdalgebra}}]\label{thm:happel}
Let $Q$ be a quiver without oriented cycles and let $q_0\in Q$ be a source. Then over a field~$k$, the reflection functors $s_{q_0}^-$ and $s_{q_0}^+$ induce a pair of inverse exact equivalences
\[
\mathbf{L}s_{q_0}^-\colon D(kQ) \rightleftarrows D(kQ')\; \colon\!\mathbf{R}s_{q_0}^+
\]
of derived categories.
\end{thm}

\begin{rmk} \label{rem:tilting-history}
Let us briefly recall how the theorem has been historically proved. While attempting to generalize the reflection functors, Auslander, Platzeck and Reiten noted in~\cite{APR79} that there exists a certain bimodule ${}_{kQ'}T_{kQ}$ such that $s_{q_0}^+$ is naturally equivalent to $\hom_{kQ'}(T,-)$. Consequently, \emph{tilting functors} were defined by Brenner and Butler in~\cite{BB80} and \emph{tilting modules} by Happel and Ringel in~\cite{HR82}. In their terminology, $(T \otimes_{kQ} -, \hom_{kQ'}(T,-))$ is a pair of tilting functors and ${}_{kQ'}T$ and $T_{kQ}$ are tilting modules.

Now the crucial observation due to Happel~\cite{happel:fdalgebra} is that the derived functor $\mathbf{R}\!\hom_{kQ'}(T,-)$ induces an (exact) equivalence between the bounded derived categories $D^b(kQ)$ and $D^b(kQ')$. Since the latter categories are precisely the categories of compact objects of $D(kQ')$ and $D(kQ)$, respectively, an easy abstract argument shows that $\mathbf{R}\!\hom_{kQ'}(T,-)$ must also be an equivalence between $D(kQ')$ and $D(kQ)$.
Finally, $(s_{q_0}^-,s_{q_0}^+)$ is checked to be a Quillen adjunction with respect to the projective model structure on $\Ch(kQ)$  (see~\cite[Definition 2.3.3 and Theorem 2.3.11]{hovey:modelcats}) and the injective model structure on $\Ch(kQ')$ (see~\cite[Definition 2.3.12 and Theorem 2.3.13]{hovey:modelcats}). In combination with the result of Happel this implies that $(s_{q_0}^-,s_{q_0}^+)$ even is a Quillen equivalence.
\end{rmk}

In this paper we obtain for oriented trees a generalization of this result which is valid in the context of an arbitrary stable derivator. More precisely, we establish as \autoref{thm:reflection} that the quivers $Q$ and $Q'$ are not only derived equivalent but actually \emph{strongly stably equivalent}. In particular, this implies variants of Happel's result for representations over an arbitrary ground ring, for quasi-coherent modules on schemes, in the differential-graded context, and in the spectral context. Moreover, these equivalences are pseudo-natural with respect to exact morphisms (for more comments on the added generality see \cite[\S5]{gst:basic}).

\begin{figure}
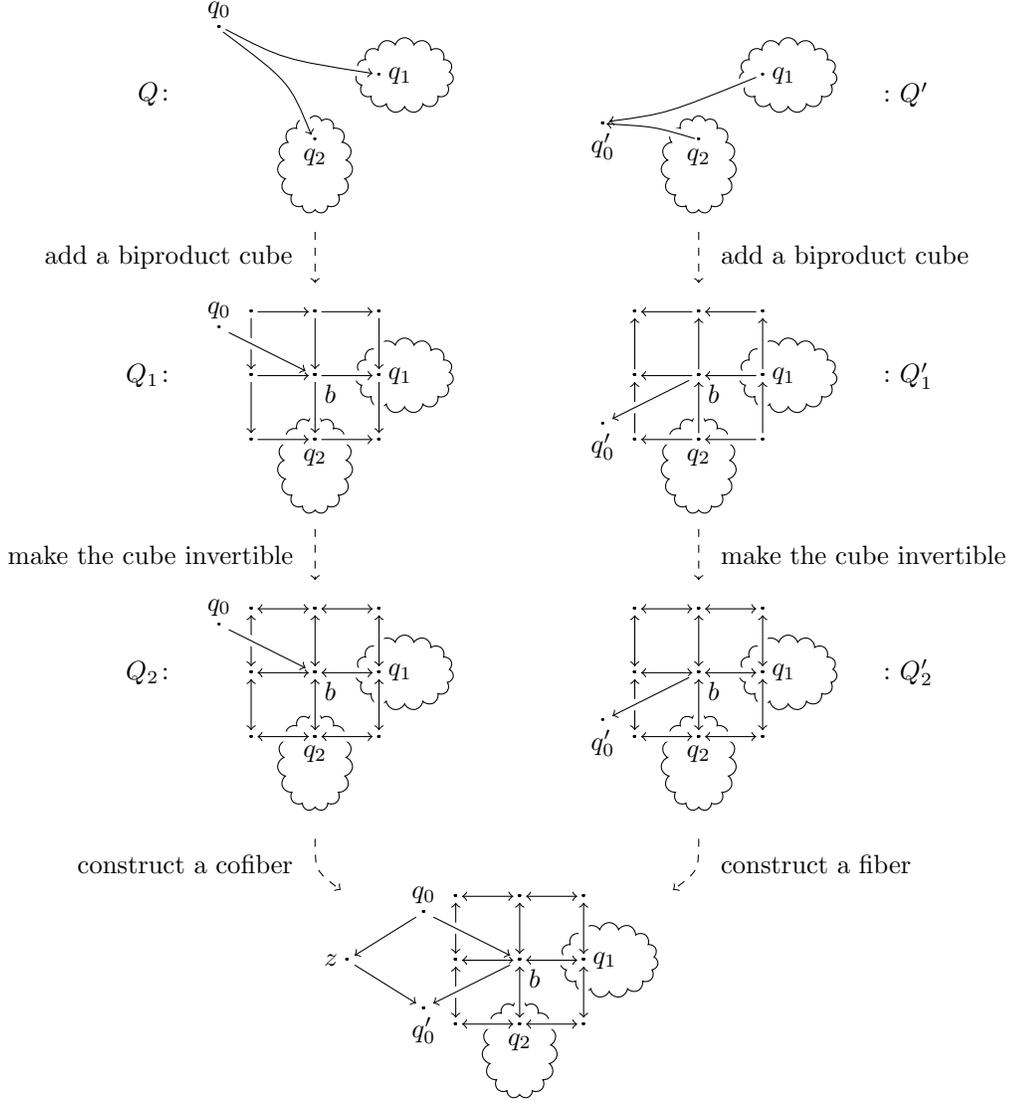

\centering
\TikzReflectionStrategy
\caption{Reflection functors using `invertible $n$-cubes', illustrated on a source/sink of valence two. }
\label{fig:reflection}
\end{figure}

The tricky point in the construction of similar reflection functors for abstract stable derivators (and hence in the context of homotopy coherent diagrams) is that one cannot simply take a morphism in such a diagram and replace it by a map in the opposite direction. Note that this is what we do in the classical case. First, in the definition of $M'_{q_0}$ we use that $M_{q_1}\oplus\ldots\oplus M_{q_n}$ enjoys the universal property of a \emph{product} (see \eqref{eq:reflection}). Then we use that it also is a \emph{coproduct} coming with canonical inclusion maps which allows us to define $M'$ on edges (see \eqref{eq:clreflect}). In order to mimic this step of `replacing' the projections of the biproduct by the inclusions in the context of abstract stable derivators we have to work a bit harder. The strategy consists of the following steps and is illustrated in \autoref{fig:reflection} (where the additional branches of the tree are depicted as `clouds').

\begin{enumerate}
\item \label{step:one} Extend the quiver $Q$ to a category by adding  $[2]^n$, an $n$-cube of length two. At the level of representations this amounts to gluing in a representation of $[2]^n$ which is a finite biproduct diagram for $M_{q_1}\oplus\ldots\oplus M_{q_n}$. This will be achieved by a variant of \autoref{prop:biproducts}.
\item \label{step:two} As in \autoref{prop:biproducts}, it will be true that the extended representations send all `length two morphisms' in the $n$-cube $[2]^n$ to isomorphisms. This suggests that the representations should be restrictions of representations of the category obtained by formally inverting all such `length two morphisms'.
Taking the cube $[2]^n$ alone, we show in \S\ref{sec:epibiprod} that this localization $q\colon[2]^n\to R^n$ is a \emph{homotopical epimorphism} (a notion we introduce in \S\ref{sec:epi}), and hence induces a fully faithful restriction functor $q^\ast\colon\D^{R^n}\to\D^{[2]^n}$. The essential image of~$q^\ast$ consists precisely of the diagrams which invert all `length two morphisms' (see \autoref{cor:invhyper} and \autoref{cor:invbiprod}).
\item \label{step:three} In order to check that this localization step interacts nicely with additional pieces of the representations attached to the cube, we introduce the concept of a \emph{one-point extension} (see \S\ref{sec:onepoint}). The key point is that one-point extensions of homotopical epimorphisms are again homotopical epimorphisms in a way that we control the corresponding essential images of the restriction functors (see \autoref{thm:onepointessim}). This allows us to inductively take into account the rest of the representations.
\item \label{step:four} Finally, we mimic the cokernel construction by adding the cofiber of the map from the source of the original quiver to the object supporting the biproduct. Observe that this final category $\tilde{Q}$ contains both the original quiver $Q$ and the reflected quiver $Q'$ as subcategories (see \autoref{fig:reflection}) -- this was the point of passing to invertible $n$-cubes. Moreover, performing dual constructions starting from either side one obtains the same representations of $\tilde{Q}$. This is carried out in~\S\ref{sec:reflection}, and shows that $Q$ and $Q'$ are strongly stably equivalent.
\end{enumerate}

\section{Homotopical epimorphisms}
\label{sec:epi}

In the theory of derivators one constantly uses the fact that Kan extensions along fully faithful functors are again fully faithful (\autoref{egs:htpy}). The following definition captures the case of fully faithful restriction functors.

\begin{defn}
Let $A,B\in\cCat$. A functor $u\colon A\to B$ is a \textbf{homotopical epimorphism} if the commutative square
\begin{equation}
  \vcenter{\xymatrix{
      A\ar[r]^u\ar[d]_u &
      B\ar[d]^=\\
      B\ar[r]_= &
      B
    }}\label{eq:hoepisq}
\end{equation}
is homotopy exact.
\end{defn}

Thus, a functor $u\colon A\to B$ is a homotopical epimorphism if and only if the counit $\epsilon\colon u_!u^\ast\to\id$ is a natural isomorphism, i.e., if and only if the restriction functor $u^\ast\colon\D(B)\to\D(A)$ is fully faithful for any derivator~\D. Moreover, an object $X\in\D(A)$ lies in the essential image of $u^\ast$ if and only if the unit $\eta\colon\id\to u^\ast u_!$ is an isomorphism on~$X$. Using right Kan extensions instead, $u$ is a homotopical epimorphism if and only if $\eta\colon\id\to u_\ast u^\ast$ is an isomorphism. In this case $X$ lies in the essential image of $u^\ast$ if and only if $\epsilon\colon u^\ast u_\ast\to\id$ is an isomorphism on~$X$.

\begin{rmk}
\begin{enumerate}
\item The definition of a homotopical epimorphism seems to depend on the notion of a derivator, but this is not the case. It is immediate from \cite[Theorem~3.16]{gps:mayer} (which in turn relies essentially on work of Heller \cite{heller:htpy} and Cisinski \cite{cisinski:presheaves}) that $u\colon A\to B$ is a homotopical epimorphism if and only if the functor 
\[
u^\ast\colon \Ho(\sset^B)\to\Ho(\sset^A)
\]
is fully faithful. Here $\sset$ is the category of simplicial sets, and the simplicial presheaf categories are endowed with projective Kan--Quillen model structures \cite{bousfieldkan}. Thus the notion `homotopical epimorphism' only depends on the classical homotopy theory of spaces.
\item In principle, one could consider the more general definition of a \emph{\D-homotopical epimorphism}, i.e., a functor $u\colon A\to B$ such that $u^\ast\colon \D(B)\to\D(A)$ is fully faithful only for a specific derivator (or similarly for a specific class of derivators). However, in this paper we will have no use for this more general concept.
\end{enumerate}
\end{rmk}

\begin{rmk}
The motivation for this terminology comes from the following relation to \emph{homological epimorphisms} in homological algebra. To start with, a \emph{ring epimorphism} $u\colon A \to B$ is simply an epimorphism in the category of rings. Ring epimorphisms are characterized by the property that the forgetful functor $u^\ast\colon \Mod(B) \to \Mod(A)$ is fully faithful; see~\cite{Sil67,Stor73,GdlP87}. A \textbf{homological epimorphism} of rings is a homological analogue of the concept and has been first studied by Geigle and Lenzing~\cite[\S4]{GL91}. It is by definition a ring homomorphism $u\colon A \to B$ such that the induced forgetful functor $u^\ast\colon D(B) \to D(A)$ is fully faithful. Recently, Pauksztello~\cite{Pauk09} and Nicolas and Saor\'{\i}n~\cite[\S4]{NS09} studied the situation when a morphism $u\colon A\to B$ of differential-graded algebras or small differential-graded categories induces a fully faithful functor $u^\ast\colon D(B) \to D(A)$ between the derived categories. They again called the corresponding notion a homological epimorphism of dg-algebras or dg-categories, respectively, and showed that it is intimately related to smashing localizations of derived categories. In view of the previous remark, our concept of homotopical epimorphism is a direct generalization of homological epimorphisms to the simplicial context.
\end{rmk}

Let us collect a few examples and closure properties of homotopical epimorphisms.

\begin{prop}\label{prop:basicepis}
~ % This is to fix the refernces to the examples
\begin{enumerate}
\item Homotopical epimorphisms are stable under compositions and contain the identities. The opposite of a homotopical epimorphism is again a homotopical epimorphism. If $u$ is naturally isomorphic to $v$, then $u$ is a homotopical epimorphism if and only if $v$ is.
\item Equivalences of categories are homotopical epimorphisms. More generally, a functor admitting a fully faithful left adjoint (i.e., a coreflective colocalization) or a fully faithful right adjoint (i.e., a reflective localization) is a homotopical epimorphism.
\item Disjoint unions and finite products of homotopical epimorphisms are again homotopical epimorphisms.
\end{enumerate}
\end{prop}
\begin{proof}
The statements in (i) are immediate. As for (ii), let $u\colon A\to B$ have a fully faithful left adjoint $v\colon B\to A$ and let $\eta\colon\id \to uv$ and $\epsilon\colon vu\to \id$ be the unit and the counit. The fully faithfulness of $v$ is equivalent to $\eta$ being an isomorphism.  For every derivator \D, we obtain an adjunction $(u^\ast,v^\ast)\colon\D^B\rightleftarrows\D^A$ with unit $\eta^\ast\colon\id\to v^\ast u^\ast$ and counit $\epsilon^\ast\colon u^\ast v^\ast\to\id$. Thus, the left adjoint $u^\ast$ is fully faithful and this establishes the claims in (ii). A coproduct of homotopical epimorphisms is again a homotopical epimorphism by (Der1). Finally, note that the restriction functor along $u_1\times u_2\colon A_1\times A_2\to B_1\times B_2$ factors as
\[
\D^{B_1\times B_2}\cong(\D^{B_1})^{B_2}\stackrel{u_2^\ast}{\to}(\D^{B_1})^{A_2}\cong(\D^{A_2})^{B_1}\stackrel{u_1^\ast}{\to}(\D^{A_2})^{A_1}\cong \D^{A_1\times A_2}.
\]
Thus, if $u_1,u_2$ both are homotopical epimorphisms then the two restriction functors are fully faithful and $u_1\times u_2$ is also a homotopical epimorphism.
\end{proof}

We will need a combinatorial detection criterion for homotopical epimorphisms which allows us to establish examples which are not covered by \autoref{prop:basicepis}. This is obtained by specializing a more general detection criterion for homotopy exact squares from \cite{gps:mayer}, which we recall next. Thus, let us again consider a natural transformation
\begin{equation}
  \vcenter{\xymatrix{
      D\ar[r]^p\ar[d]_q \drtwocell\omit{\alpha} &
      A\ar[d]^u\\
      B\ar[r]_v &
      C
    }}
\label{eq:htpyexact}
\end{equation}
living in a square of small categories.

\begin{defn}
For $a\in A$, $b\in B$, and $\gamma\colon u(a)\to v(b)$, let
  \((a/D/b)_\gamma\)
  be the category consisting of triples $(d\in D, a\xto{\phi} p(d), q(d) \xto{\psi}b)$ such that 
\[
\gamma= v\psi \circ \alpha_d \circ u\phi \colon u(a)\to up(d)\to vq(d)\to v(b).
\]
Morphisms are morphism in $D$ making the obvious triangles in $A$ and $B$ commute.
\end{defn}

The category \((a/D/b)_\gamma\) consists of certain two-sided factorizations of~$\gamma$ through components of the natural transformation~$\alpha$. We will have a use for this category both in this generality but also under the simplifying assumptions that $\alpha,u,$ and $v$ are identities.

Recall that associated to a small category~$E$ there is the simplicial set $N(E),$ the \emph{nerve} of~$E$, which in degree~$n$ is the set of strings of $n$ composable morphisms in~$E.$ 

\begin{thm}[{\cite[Theorem~3.16]{gps:mayer}}]\label{thm:hoexchar}
The square~\eqref{eq:htpyexact} is homotopy exact if and only if the nerve of $(a/D/b)_\gamma$ is weakly contractible for all $a$, $b$, and $\gamma$.
\end{thm}

Note that the converse direction of this theorem implies that the notion of homotopy exact squares depends only on the classical homotopy theory of spaces (see \cite[\S3]{gps:mayer} for more details). For later reference, we now specialize this combinatorial criterion to the context of homotopical epimorphisms (see \eqref{eq:hoepisq}).

\begin{defn}
Let $u\colon A\to B$ be a functor and $\gamma\colon b_1\to b_2$ be a morphism in~$B$. Then \((b_1/A/b_2)_\gamma\) is the category of factorizations of $\gamma$ as $b_1\to u(a)\to b_2.$ A morphism $(a,b_1\to u(a)\to b_2)\to (a',b_1\to u(a')\to b_2)$ is a morphism $a\to a'$ in~$A$ making the obvious diagram in $B$ commute.
\end{defn}

\begin{thm}\label{thm:detectepi}
A functor $u\colon A\to B$ is a homotopical epimorphism if and only if the categories $(b_1/A/b_2)_\gamma$ have weakly contractible nerves for all $b_1,b_2\in B$ and all $\gamma\colon b_1\to b_2.$ If this is the case, then $u^\ast\colon\D^B\to\D^A$ induces an equivalence onto its essential image which consists precisely of the objects $X\in\D^A$ such that the adjunction unit $\eta\colon X\to u^\ast u_!(X)$ or equivalently the adjunction counit $\epsilon\colon u^\ast u_\ast(X)\to X$ is an isomorphism.
\end{thm}
\begin{proof}
The first part is immediate from \autoref{thm:hoexchar}. The characterization of the essential image of the restriction functor is immediate from the fact that there are adjunctions $(u_!,u^\ast)$ and $(u^\ast,u_\ast)$ with $u^\ast$ being fully faithful.
\end{proof}

\begin{cor}\label{cor:piepi}
The projection functor $\pi_A\colon A\to\bbone$ is a homotopical epimorphism if and only if the nerve $N(A)$ is weakly contractible.
\end{cor}
\begin{proof}
It suffices to observe that the category $(\ast/A/\ast)_\id$ is isomorphic to $A$. Since this is the only case to be considered the result is immediate.
\end{proof}

Thus, if $N(A)$ is weakly contractible, then any morphism $\pi_A^\ast X\to\pi_A^\ast Y$ can be written as $\pi_A^\ast(f)$ for a unique map $f\colon X\to Y$ in the underlying category $\D(\bbone)$.

\begin{rmk}
Recall that a functor $u\colon A\to B$ is a \textbf{homotopy equivalence} if the canonical mate
\begin{equation}\label{eq:hoeqv}
(\pi_A)_! (\pi_A)^* \cong (\pi_B)_! u_! u^* (\pi_B)^* \to (\pi_B)_! (\pi_B)^*
\end{equation}
is an isomorphism in any derivator. Intuitively, a functor $A\to B$ is a homotopy equivalence if and only if it tells us that homotopy colimits of $A$-shaped constant diagrams and $B$-shaped constant diagrams are canonically isomorphic. Heller~\cite{heller:htpy} and Cisinski~\cite{cisinski:presheaves} showed that a functor $u\colon A\to B$ is a homotopy equivalence if and only if $N(u)\colon N(A)\to N(B)$ is a weak homotopy equivalence. Since vertical pastings of homotopy exact squares are again homotopy exact, the diagram
\[
\xymatrix{
A\ar[r]^-u\ar[d]_-u&B\ar[r]^-{\pi_B}\ar[d]^-=&\bbone\\
B\ar[d]_-{\pi_B}\ar[r]_-=&B\ar[d]^-{\pi_B}&\\
\bbone\ar[r]_-=&\bbone&
}
\]
tells us that homotopical epimorphisms are homotopy final functors. Recall from \autoref{egs:htpy} that these functors tell us something about homotopy colimits of not necessarily constant diagrams. In particular, by only considering constant diagrams we see that homotopy final functors are homotopy equivalences. As a summary, there are the following relations between these various notions:
\begin{align}
& u\colon A\to B \text{ is a homotopical epimorphism}\\
\Rightarrow\quad & u\colon A\to B\text{ is homotopy final}\\
\Rightarrow\quad & u\colon A\to B \text{ is a homotopy equivalence}\\
\Leftrightarrow\quad & N(u)\colon N(A)\to N(B) \text{ is a weak homotopy equivalence}
\end{align}
\end{rmk}

The notion `homotopical epimorphism' certainly deserves to be studied more systematically. Here however we only develop what is necessary for our approach to abstract reflection functors: in \S\ref{sec:epibiprod} we establish a key example which allows us to give a more symmetric description of biproduct diagrams and in \S\ref{sec:onepoint} we study the behavior of homotopical epimorphisms with respect to one-point extensions.

\section{Finite biproducts via invertible $n$-cubes}
\label{sec:epibiprod}

The aim of this section is to show that the `passage from the $n$-cube to the invertible $n$-cube' is a homotopical epimorphism, and to understand the essential image of the associated restriction functor (see \autoref{cor:invhyper}). To formalize this, we begin by the following definition.

\begin{defn}
Let $R$ be the category obtained from $[2]=(0\to 1\to 2)$ by freely inverting the morphism $0\to 2.$ The localization functor is denoted $p\colon [2]\to R.$
\end{defn}

Thus, the category $R$ has objects 0,1, and 2 and both objects 0 and 2 are zero objects. Moreover, the object $1\in R$ has a non-identity endomorphism $t\colon 1\to 1$ which is given by the zero map and hence is idempotent. A picture of the category suppressing the identity morphisms is
\medskip
\begin{equation}
\vcenter{
\xymatrix{
0\ar@<0.5ex>[r]&1\ar@<0.5ex>[r]\ar@<0.5ex>[l]\ar@(ul,ur)[]&2\ar@<0.5ex>[l],
}
}
\label{eq:category-R}
\end{equation}
in which the unlabeled loop is the idempotent $t\colon 1\to 1$ and in which the maps $0\to 2$ and $2\to 0$ are inverse isomorphisms. Thus, in total there are ten morphisms in $R$. Note that this category is equivalent to a category freely generated by a section-retraction pair (which has only two objects and five morphisms).

Since the following proposition is essential to our approach to reflection functors in~\S\ref{sec:reflection} and since this is the first instance of a proof in this paper in which homotopy exact squares and the calculus of mates (see \eqref{eq:hoexmate1'} and \eqref{eq:hoexmate2'}) play a key role, we include a reasonably complete proof. Note that the result is not formal in that~$p$ is neither a reflective localization nor a coreflective colocalization, and the result is hence not covered by \autoref{prop:basicepis}.

\begin{prop}\label{prop:pepi}
The localization functor $p\colon [2]\to R$ is a homotopical epimorphism. For every derivator \D, $p^\ast\colon\D^R\to\D^{[2]}$ induces an equivalence onto the full subderivator of $\D^{[2]}$ spanned by the objects $X$ such that $X_0\to X_2$ is an isomorphism.
\end{prop}
\begin{proof}
In order to check that $p$ is a homotopical epimorphism, we apply the combinatorial detection criterion \autoref{thm:detectepi}. In our situation this amounts to showing that the ten different categories $(r_1/[2]/r_2)_\gamma$ have weakly contractible nerves. Let us content ourselves by giving the details for the case of the non-identity idempotent $\gamma=t\colon 1\to 1.$ In this case, $(1/[2]/1)_t$ is freely generated by the graph
\[
\xymatrix{
&(0,1\to 0\to 1)\ar[rd]^-{(0\to 1)}\ar[d]_-{(0\to 1)}&\\
(1,1\stackrel{=}{\to}1\stackrel{t}{\to}1)\ar[dr]_-{(1\to 2)}&(1,1\stackrel{t}{\to}1\stackrel{t}{\to}1)\ar[d]^-{(1\to 2)}& (1,1\stackrel{t}{\to}1\stackrel{=}{\to}1)\\
&(2,1\to 2 \to 1)&
}
\]
where the undecorated morphisms are the unique maps in $[2]$ and $R$, respectively. Thus the nerve of this category is weakly contractible. The remaining nine cases are similar: one checks that the nerves of the respective categories are certain very small oriented trees (with at most three edges in the remaining cases) and that the nerves are hence also weakly contractible. Thus, \autoref{thm:detectepi} implies that $p\colon[2]\to R$ is a homotopical epimorphism.

Let us now describe the essential image of $p^\ast.$ As the morphism $0\to 2$ is invertible in $R$ it is immediate that $X_0\to X_2$ is an isomorphism for all $X$ in the essential image of $~p^\ast$. The converse is more involved. Since $p$ is a homotopical epimorphism, we are given an adjunction $(p_!,p^\ast)$ with a fully faithful right adjoint. Thus, $X$ lies in the essential image of $p^\ast$ if and only if the adjunction unit $\eta\colon X\to p^\ast p_!(X)$ is an isomorphism. Note that the adjunction unit is the canonical mate \eqref{eq:hoexmate1'} associated to the square on the right in
\begin{equation}
\vcenter{\xymatrix{
\bbone\ar[r]^-i\ar[d]_-=\drtwocell\omit{\id}&[2]\ar[r]^=\ar[d]_= \drtwocell\omit{\id}&[2]\ar[d]^p\\
\bbone\ar[r]_-i&[2]\ar[r]_p&R.
}}
\end{equation}
By (Der2), isomorphisms are detected pointwise and it is hence enough to show that the mate is an isomorphism for each $i=0,1,2$. Given $i=0,1,2$ one uses the compatibility of mates with pasting (see \autoref{egs:htpy}) together with the homotopy exactness of the above constant squares on the left 
to conclude that~$\eta_i$ is an isomorphism at~$X$ if and only if the mate associated to 
\begin{equation}
\vcenter{\xymatrix{
\bbone\ar[r]^i\ar[d]_= \drtwocell\omit{\id}&[2]\ar[d]^p\\
\bbone\ar[r]_-i&R
}}\label{eq:etai}
\end{equation}
is an isomorphism at~$X$.

We now re-express this statement for $i=0$ showing that it is satisfied under the additional assumption that $X_0\to X_2$ is an isomorphism. A homotopy finality argument tells us that the pasting
\begin{equation}
\vcenter{\xymatrix{
\bbone\ar[r]^-2\ar[d]&(p/0)\ar[r]\ar[d] \drtwocell\omit{}&[2]\ar[d]^p\\
\bbone\ar[r]&\bbone\ar[r]_-0&R
}}
\end{equation}
is homotopy exact, and the mate of the pasting is hence an isomorphism for an arbitrary~$X$. Note that this is not yet the diagram~\eqref{eq:etai} used to calculate~$\eta_0$ but that~\eqref{eq:etai} is obtained from it by a vertical pasting as in
\begin{equation}
\vcenter{\xymatrix{
\bbone\ar[r]^-0\ar[d]\drtwocell\omit{}&[2]\ar[d] \\
\bbone\ar[r]^-2\ar[d]\drtwocell\omit{}&[2]\ar[d]^p\\
\bbone\ar[r]_-0&R.
}}
\end{equation}
Now the fact that $X_0\to X_2$ is an isomorphism says precisely that the mate associated to the top square is an isomorphism when applied to~$X$.

To show that $\eta_2$ is invertible is simpler than the previous case in that it suffices to use a homotopy finality argument (see \autoref{egs:htpy}) as expressed by the pasting situation
\begin{equation}
\vcenter{\xymatrix{
\bbone\ar[r]^2\ar[d]&[2]\ar[r]\ar[d] \drtwocell\omit{}&[2]\ar[d]^p\\
\bbone\ar[r]&\bbone\ar[r]_-2&R.
}}
\end{equation}
Since this pasting is the same as the square \eqref{eq:etai} for $i=2$ we conclude that $\eta_2$ is an isomorphism without any further assumption on~$X$.

The remaining case of $i=1$ is the trickiest case, in which we start by considering the homotopy exact square
\begin{equation}
\vcenter{\xymatrix{
(p/1)\ar[r]\ar[d] \drtwocell\omit{}&[2]\ar[d]^p\\
\bbone\ar[r]_-1&R
}}
\end{equation}
given by (Der4). The category $(p/1)$ is freely generated by the graph
\begin{equation}
\vcenter{
\xymatrix{
(0,0\to 1)\ar[d]_-{(0\to 1)}\ar[r]^-{(0\to 1)}&(1,1\stackrel{t}{\to}1)\ar[r]^-{(1\to 2)}& (2,2\to 1)\\
(1,1\stackrel{=}{\to} 1)&&
}}\label{eq:slice1}
\end{equation}
where again the undecorated morphisms are the unique maps in $[2]$ or $R$, respectively. The inclusion $\ulcorner\to(p/1)$ which does not hit the object $(1,1\stackrel{t}{\to}1)$ is a right adjoint and hence homotopy final by \autoref{egs:htpy}. Using the compatibility of homotopy exact squares with respect to pasting again, we conclude that the pasting
\begin{equation}
\vcenter{\xymatrix{
\ulcorner\ar[r]\ar[d]&(p/1)\ar[r]\ar[d] \drtwocell\omit{}&[2]\ar[d]^p\\
\bbone\ar[r]&\bbone\ar[r]_-1&R
}}
\end{equation}
of the two squares to the right in  
\begin{equation}
\vcenter{\xymatrix{
\bbone\ar[r]^-1\ar[d]&[1]\ar[r]^-j\ar[d]&\ulcorner\ar[r]\ar[d]&(p/1)\ar[r]\ar[d] \drtwocell\omit{}&[2]\ar[d]^p\\
\bbone\ar[r]&\bbone\ar[r]&\bbone\ar[r]&\bbone\ar[r]_-1&R
}}
\end{equation}
is homotopy exact. In the above diagram $j\colon[1]\to\ulcorner$ is the vertical inclusion. Using the explicit description of the slice category \eqref{eq:slice1} we see that the pasting agrees with \eqref{eq:etai} for $i=1$. Since $1\in[1]$ is a terminal object, the square to the very left is homotopy exact by \autoref{egs:htpy}. Thus, using once more the compatibility of mates with pasting, it remains to show the following: if $Y\in\D(\ulcorner)$ is such that the horizontal map $Y_{0,0}\to Y_{1,0}$ is an isomorphism then the canonical mate $\colim_{[1]}j^\ast Y\to\colim_\ulcorner Y$ associated to the second square from the left is an isomorphism.

To show this we note that $j_!\colon\D([1])\to\D(\ulcorner)$ is fully faithful and that the essential image consists precisely of those $Y\in\D(\ulcorner)$ such that $Y_{0,0}\to Y_{1,0}$ is an isomorphism (this is the obvious variant of \cite[Prop.~3.12.(1)]{groth:ptstab}). The uniqueness of left adjoints yields a canonical isomorphism $\colim_{[1]}\cong\colim_\ulcorner j_!$ and this isomorphism makes the following diagram 
\[
\xymatrix{
\colim_{[1]}j^\ast j_!\ar[r]&\colim_\ulcorner j_!\\
\colim_{[1]}\ar[ru]_-\cong\ar[u]^-\eta
}
\]
commutative. Since $j$ is fully faithful, $\eta\colon\id\to j^\ast j_!$ is invertible and so is the horizontal morphism in this diagram. The above description of the essential image of $j_!$ concludes the proof.
\end{proof}

There is the following variant of the proposition for finitely many coordinates.

\begin{cor}\label{cor:invhyper}
The functor $q=p^{\times n}\colon [2]^{\times n}\to R^{\times n}$ is a homotopical epimorphism and $q^\ast\colon \D^{R^{\times n}}\to \D^{[2]^{\times n}}$ induces an equivalence onto the full subderivator of $\D^{[2]^{\times n}}$ spanned by all diagrams $X$ such that 
\[
X_{i_1,\ldots,i_{k-1},0,i_{k+1},\ldots,i_n}\to X_{i_1,\ldots,i_{k-1},2,i_{k+1},\ldots,i_n}
\]
is an isomorphism for all $i_1,\ldots,i_{k-1},i_{k+1},\ldots,i_n$ and $k$.
\end{cor}
\begin{proof}
The fact that $q$ is a homotopical epimorphism follows immediately from \autoref{prop:pepi} and the stability of homotopical epimorphisms under finite products (see \autoref{prop:basicepis}). The essential image of $q^\ast$ can be described inductively using \autoref{prop:pepi} again.
\end{proof}

It is immediate to see that this equivalence can be restricted to equivalences of full subderivators spanned by diagrams satisfying certain vanishing or exactness conditions. Let \D be a stable derivator and let $\D^{R^n,\mathrm{ex}}$ be the full subprederivator of $\D^{R^n}$ spanned by all diagrams such that 
\begin{enumerate}
\item all subcubes are strongly bicartesian and
\item the values at all corners are trivial.
\end{enumerate}

\begin{cor}\label{cor:invbiprod}
Let \D be a stable derivator. For every $n\geq 2$ there is a pseudo-natural equivalence $\D\times\ldots\times\D\simeq\D^{n\cdot\bbone}\to\D^{R^n,\mathrm{ex}}$.
\end{cor}
\begin{proof}
By axiom (Der1) there is a pseudo-natural equivalence $\D\times\ldots\times\D\simeq\D^{n\cdot\bbone}$. By \autoref{prop:biproducts} there is a pseudo-natural equivalence $\D^{n\cdot\bbone}\simeq\D^{[2]^n,\mathrm{ex}}$. The explicit descriptions of $\D^{[2]^n,\mathrm{ex}}$ and $\D^{R^n,\mathrm{ex}}$ together with \autoref{cor:invhyper} imply that we also have a pseudo-natural equivalence between these two derivators which concludes the proof.
\end{proof}

Thus, given a stable derivator, this corollary makes precise a way to encode coherent versions of finite biproduct diagrams in a completely symmetric fashion. Recall from \S\ref{sec:reflectrep} that this is roughly step \ref{step:two} in our strategy for the construction of reflection functors in \S\ref{sec:reflection}.

\section{One-point extensions}
\label{sec:onepoint}

Given an arbitrary tree $Q$ and a source $q_0\in Q$ of valence~$n$, we would like to glue an $n$-cube $[2]^n$ into the quiver, invert the cube, and then show that this localization functor is a homotopical epimorphism. Moreover, we want to show that the restriction functor induces an equivalence onto the derivator of all diagrams inverting the length two morphisms. Instead of trying to show this directly, we proceed differently. The following simple formalism of one-point extensions of small categories allows us to inductively reduce this more general context to the situation we already studied in~\S\ref{sec:epibiprod}. 

\begin{defn}
Let $A$ and $A'$ be small categories. The category $A'$ is a \textbf{one-point extension} of $A$ if there is a pushout diagram
\begin{equation}
\vcenter{
\xymatrix{
\bbone\ar[r]^-a\ar[d]_-i&A\ar[d]^-{j_A}\\
[1]\ar[r]&A'.
}
}
\label{eq:defone}
\end{equation}
\end{defn}

Thus, we simply attach a morphism to $A$ by identifying either its source or its target with a given object in $A$. Note that $A'$ comes with two full subcategories $A,A'-A\cong\bbone$ which are \emph{convex} in the sense that any morphism in these subcategories cannot factor through an object in the complement. 

If $A'$ is a one-point extension of $A$ as in \eqref{eq:defone} and if $u\colon A\to B$ is a functor, then there is a canonical functor $u'\colon A'\to B'$, where $B'$ is the corresponding one-point extension of $B$ at $u(a)$. More formally, there is the following diagram
\begin{equation}\label{eq:one-point-basics}
\vcenter{
\xymatrix{
\bbone\ar[r]^-a\ar[d]_-i&A\ar[r]^-u\ar[d]_-{j_A}&B\ar[d]^-{j_B}\ar@{}[rrd]|{=}&&
\bbone\ar[r]^-{u(a)}\ar[d]_-i&B\ar[d]^-{j_B}\\
[1]\ar[r]_-f&A'\pushoutcorner\ar[r]_-{u'}&B'&&
[1]\ar[r]_-{u'(f)}&B'.\pushoutcorner
}
}
\end{equation}
The cancellation property of pushout squares in $\cCat$ implies that also the square in the middle is a pushout square. In this situation, we say that $u'\colon A'\to B'$ is a \textbf{one-point extension} of $u\colon A\to B$. 

\begin{rmk}
Our terminology is again based on a similar construction common in representation theory; see~\cite[\S III.2]{ARS97}. If $A = kQ/I$ is a finite dimensional path algebra and $M \in \Mod(A)$, a \textbf{one-point extension} is the triangular matrix ring $B = \left(\begin{smallmatrix}k & 0 \\ M & A\end{smallmatrix}\right)$ with the obvious matrix multiplication. It is easy to see that $B$ is of the form $B = kQ'/I'$, where $Q'$ is obtained from $Q$ by adding a new source and arrows from it to vertices of $Q$. If $M$ is an indecomposable projective $A$-module, then $Q'$ is obtained from $Q$ by adding precisely one arrow from the source and no new relations. Thus, this is precisely a linearized version of our one-point extension defined above.
\end{rmk}

We collect a few examples and properties of one-point extensions.

\begin{egs}\label{egs:one-point-basics}
~ % This is to fix the refernces to the examples
\begin{enumerate}
\item Let $T$ be a finite tree and let $T'\subseteq T$ be a non-empty connected subtree. Then $T$ can be obtained from $T'$ by iterated one-point extensions.
\item Let $P'$ be a poset with a unique minimal element $-\infty$, and let $P'$ be the poset obtained from $P$ by adjoining a new element $-\infty-1$ and with order relation generated by the one on $P$ and $-\infty-1\leq -\infty$. (Thus, if we consider $P'$ as a category, then $P=(P')^\lhd$.) Then 
\[
\xymatrix{
\bbone\ar[rr]^-{-\infty}\ar[d]_-1&&P'\ar[d]\\
[1]\ar[rr]_-{-\infty-1\leq-\infty}&&P
}
\]
exhibits $P$ as a one-point extension of $P'$.
\item If $u\colon A \to B$ is a functor in $\cCat$ and $j_A\colon A \to A'$ is a one-point extension, then the pushout $u'\colon A' \to B'$ of $u$ along $j_A$, 
\begin{equation}
\vcenter{
\xymatrix{
A\ar[r]^-{j_A}\ar[d]_-u&A'\ar[d]^-{u'}\\
B\ar[r]_-{j_B}&B',
}
}
\label{eq:onepointsquare}
\end{equation}
is again a one-point extension of $u$. In fact, this is the glueing property of pushout squares in $\cCat$ applied to \eqref{eq:one-point-basics}.
\end{enumerate}
\end{egs}

\begin{lem}\label{lem:onepoint-ad}
Let $u'\colon A'\to B'$ be a one-point extension of $u\colon A\to B$.
\begin{enumerate}
\item If $A'$ is obtained by attaching a morphism to its target, then $j_A\colon A\to A'$ admits a left adjoint $l_A$, i.e., $A$ is a reflective localization of $A'$. The functors~$l_A$ and $l_B$ can be chosen such that $u\circ l_A=l_B\circ u'\colon A'\to B$.
\item If $A'$ is obtained by attaching a morphism to its source, then $j_A\colon A\to A'$ admits a right adjoint $r_A$, i.e., $A$ is a coreflective colocalization of $A'$. The functors~$r_A$ and $r_B$ can be chosen such that $u\circ r_A=r_B\circ u'\colon A'\to B$.
\end{enumerate}
In particular, the nerves $N(A)$ and $N(A')$ are homotopy equivalent.
\end{lem}
\begin{proof}
We consider the case in (i). The left adjoint $l_A$ can be chosen to be the identity on $A\subseteq A'$. If the new morphism in $A'$ is given by $f\colon a_{-1}\to a_0$ with $a_0\in A$, then we set $l_A(a_{-1})=a_0$. Moreover, any map $\gamma\colon a_{-1}\to a$ with $a\in A$ factors uniquely as $a_{-1}\stackrel{f}{\to} a_0\stackrel{\tilde{\gamma}}{\to} a$ and we set $l_A(\gamma)=\tilde{\gamma}$. One checks that this defines the desired left adjoint to the fully faithful $j_A\colon A\to A'$, and the relation $u\circ l_A= l_B\circ u'$ is immediate. Since natural transformations induce simplicial homotopies between the induced maps on nerves it is immediate that any adjunction gives rise to a simplicial homotopy equivalence.
\end{proof}

\begin{prop}\label{prop:onepointepi}
A one-point extension $u'\colon A'\to B'$ of a homotopical epimorphism $u\colon A\to B$ is a homotopical epimorphism.
\end{prop}
\begin{proof}
We establish this result by applying the combinatorial detection principle for homotopical epimorphisms (\autoref{thm:detectepi}). Let us assume that $A'$ is obtained from $A$ by adding an object $a_{-1}$ together with a morphism $f\colon a_{-1}\to a_0$ (the case of the other orientation is similar). Then in $B'$ there is the corresponding morphism $u'(f)\colon b_{-1}\to b_0=u(a_0)$ describing $B'$ as a one-point extension of $B$. We note that the full subcategories $\{b_{-1}\},B\subseteq B'$ are convex. By \autoref{thm:detectepi} we have to show that the categories $(b_1/A'/b_2)_\gamma$ have weakly contractible nerves for all morphisms $\gamma\colon b_1\to b_2$ in $B'$. There are the following four cases to be considered.
\begin{enumerate}
\item Let $b_1,b_2\in B\subseteq B'$. Then the convexity of $B$ gives us an isomorphism of categories $(b_1/A'/b_2)_\gamma\cong(b_1/A/b_2)_\gamma$ where $(b_1/A/b_2)_\gamma$ is defined using the functor $u\colon A\to B$. But since $u\colon A\to B$ is a homotopical epimorphism this latter category has a weakly contractible nerve.
\item If $b_1=b_2=b_{-1}$ then there is only the identity map $\gamma=\id$ and it is immediate that $(b_{-1}/A'/b_{-1})_\id\cong \bbone$ so that the nerve is weakly contractible.
\item If $b_1\in B$ and $b_2=b_{-1}$ then there is no map $\gamma$ to be considered.
\item The remaining case is given by $b_1=b_{-1}$ and $b_2\in B$. By definition of $B'$ any such $\gamma\colon b_{-1}\to b_2$ can be uniquely factored as $u'(f)\colon b_{-1}\to b_0$ followed by a map $\tilde\gamma\colon b_0\to b_2$ in $B$. We leave it to the reader to check that this implies that the category $(b_{-1}/A'/b_2)_\gamma$ is obtained by a one-point extension from $(b_0/A/b_2)_{\tilde\gamma}$ which in turn is defined using the functor $u\colon A\to B$. The new object is given by $(a_{-1},\id\colon b_{-1}\to b_{-1},\gamma\colon b_{-1}\to b_2)$, and it is attached to $(a_0,\id\colon b_0\to b_0,\tilde\gamma\colon b_0\to b_2)$ by means of the map $f$. But since $u\colon A\to B$ is a homotopical epimorphism, $(b_0/A/b_2)_{\tilde\gamma}$ has a weakly contractible nerve and so the same is true for $(b_{-1}/A'/b_2)_\gamma$ by \autoref{lem:onepoint-ad}.
\end{enumerate} 
Thus in all cases the categories $(b_1/A'/b_2)_\gamma$ have weakly contractible nerves, and we are done by \autoref{thm:detectepi}.
\end{proof}

The next aim is to describe the essential image of $u'^\ast$ in terms of the one of $u^\ast$ and for that purpose we show that the square \eqref{eq:onepointsquare} is homotopy exact. We allow ourselves to re-emphasize that, in general, one has to be careful about what one means by saying that a \emph{commutative square} is homotopy exact. (Recall that there is such a potential source of confusion in the case of pullback squares involving Grothendieck (op)fibrations (see \cite[\S1.3]{groth:ptstab})). Even at the risk of being picky, let us consider the following squares
\[
\xymatrix{
A\ar[r]^-{j_A}\ar[d]_-u\drtwocell\omit{\id}&A'\ar[d]^-{u'}&&
A\ar[r]^-{j_A}\ar[d]_-u&A'\ar[d]^-{u'}\\
B\ar[r]_-{j_B}&B',&&
B\ar[r]_-{j_B}&B',\ultwocell\omit{\id}
}
\]
and make the following case distinction.

\begin{prop}\label{prop:oneexact}
Let $u'\colon A'\to B'$ be a one-point extension of $u\colon A\to B$.
\begin{enumerate}
\item The above square on the left is homotopy exact, i.e., the canonical mate $u_!j_A^\ast\to j_B^\ast u'_!$ or equivalently the canonical mate $(u')^\ast (j_B)_\ast\to (j_A)_\ast u^\ast $ is an isomorphism.
\item The above square on the right is homotopy exact, i.e., the canonical mate $j_B^\ast u'_\ast\to u_\ast j_A^\ast$ or equivalently the canonical mate $(j_A)_! u^\ast \to (u')^\ast (j_B)_! $ is an isomorphism.
\end{enumerate}
\end{prop}
\begin{proof}
We take care of the case that $A'$ is obtained from $A$ by attaching a morphism $f\colon a_{-1}\to j_A(a_0)$ for some object $a_0\in A$, and hence similarly for $B'$ with the morphism $u'(f)\colon b_{-1}\to j_B(b_0)$. (Note that this is enough since passing to opposite categories is covariant in functors but contravariant in transformations. Hence the homotopy exactness of the square on the left in case we attach a morphism to its target implies the homotopy exactness of the square on the right in the case where we attach a morphism to its source.)

We want to prove the results and we begin with statement (i), i.e., we show that the canonical mate $(u')^\ast (j_B)_\ast\to (j_A)_\ast u^\ast$ is an isomorphism. Since the functors $j_A\colon A\to A'$ and $a'_{-1}\colon\bbone\to A'$ are jointly surjective on objects, by (Der2) it suffices to show that the respective restrictions of the above canonical mate are isomorphisms. In the case of $j_A$, we consider the following diagram
\[
\xymatrix{
A\ar[r]^-\id\ar[d]_-\id\drtwocell\omit{\id}&A\ar[d]^-{j_A}&&A\ar[r]^-\id\ar[d]_-u\drtwocell\omit{\id}&A\ar[d]^-u\\
A\ar[r]^-{j_A}\ar[d]_-u\drtwocell\omit{\id}&A'\ar[d]^-{u'}&=&B\ar[r]^-\id\ar[d]_-\id\drtwocell\omit{\id}&B\ar[d]^-{j_B}\\
B\ar[r]_-{j_B}&B'&&B\ar[r]_-{j_B}&B'.
}
\]
Since these two pastings agree and since the three additional squares are homotopy exact ($j_A$ and $j_B$ are fully faithful), the functoriality of mates with pasting implies that $(j_A)^\ast(u')^\ast (j_B)_\ast\to (j_A)^\ast(j_A)_\ast u^\ast$ is an isomorphism.

For (i) it remains to show that also $(a'_{-1})^\ast(u')^\ast (j_B)_\ast\to (a'_{-1})^\ast(j_A)_\ast u^\ast$ is an isomorphism. To this end we consider the diagrams
\[
\xymatrix{
\bbone\ar[r]^-\id\ar[d]_-{a_0}\drtwocell\omit{}&\bbone\ar[d]^-{a_{-1}}\ar@{}[drr]|{=}&&\bbone\ar[r]^-\id\ar[d]_-{b_0}
\drtwocell\omit{}&\bbone\ar[d]^-{b_{-1}}\ar@{}[drr]|{=}&&\bbone\ar[r]^-\id\ar[d]\drtwocell\omit{}&\bbone\ar[d]\\
A\ar[r]^-{j_A}\ar[d]_-u\drtwocell\omit{\id}&A'\ar[d]^-{u'}&&B\ar[r]_-{j_B}&B' &&(b_{-1}/j_B)\ar[r]\ar[d]\drtwocell\omit{}&\bbone\ar[d]^-{b_{-1}}\\
B\ar[r]_-{j_B}&B'&&& &&B\ar[r]_-{j_B}&B'
}
\]
and the functoriality of mates with pasting implies that it suffices to show that the square in the middle is homotopy exact. But this square in turn is equal to the pasting on the right, which is given by a slice square and the square expressing the homotopy initiality of the functor $\bbone\to(b_{-1}/j_B)$ classifying the initial object $(b_0,b_{-1}\to b_0)$ (\autoref{egs:htpy}). 

It remains to establish (ii) which we do by showing that the canonical mate $(j_A)_!u^\ast\to(u')^\ast (j_B)_!$ is an isomorphism. Using again that $j_A\colon A\to A'$ and $a'_{-1}\colon \bbone\to A'$ are jointly surjective on objects, it suffices to consider the corresponding restrictions of this mate. In the case of the restriction along $j_A$ it suffices to swap the orientations of the identity transformations in the argument for the corresponding case in (i). Finally, the restriction along $a'_{-1}$ is necessarily an isomorphism since $j_A,j_B$ are inclusions of cosieves and the corresponding left Kan extensions are hence left extensions by initial objects (use the `unpointed variant' of \autoref{lem:extbyzero}). 
\end{proof}

We just showed that the square \eqref{eq:onepointsquare} is homotopy exact, independently of the orientation of the identity transformation and also independently of the orientation of the morphism attached in the one-point extensions.

\begin{thm}\label{thm:onepointessim}
Let $\D$ be a derivator and let $u'\colon A'\to B'$ be a one-point extension of a functor $u\colon A \to B$ in $\cCat$,
\[
\xymatrix{A\ar[r]^-{j_A}\ar[d]_-u&A'\ar[d]^-{u'}\\
B\ar[r]_-{j_B}&B'.
}
\]
\begin{enumerate}
\item If $u$ is a homotopical epimorphism, so is $u'$ and $X'\in\D^{A'}$ lies in the essential image of $u'^\ast\colon\D^{B'}\to\D^{A'}$ if and only if $j_A^\ast(X')$ lies in the essential image of $u^\ast\colon\D^B\to\D^A$.
\item If $u$ is fully faithful, so is $u'$. Moreover, $Y'\in\D^{B'}$ lies in the essential image of $u'_\ast\colon\D^{A'}\to\D^{B'}$ if and only if $j_B^\ast(Y')$ lies in the essential image of $u_\ast\colon\D^A\to\D^B$, and $Y'$ lies in the essential image of $u'_!\colon\D^{A'}\to\D^{B'}$ if and only if $j_B^\ast(Y')$ lies in the essential image of $u_!\colon\D^A\to\D^B$.
\end{enumerate}

\end{thm}
\begin{proof}
(i) We first give the details for the case that $A'$ is obtained from $A$ by attaching a map $a_{-1}\to a_0$ with $a_0\in A$, and then say a few words about the other case. Let $\eta\colon\id\to u^\ast u_!$ and $\eta'\colon\id\to u'^\ast u'_!$ be the adjunction units. Since $u$ is a homotopical epimorphism, the same is true for $u'$ by \autoref{prop:onepointepi}. We thus have to show that $\eta'$ is an isomorphism on $X'\in\D^{A'}$ if and only if $\eta$ is an isomorphism on the restriction $j_A^\ast(X')\in\D^A$.

We begin by showing that $\eta'$ is an isomorphism if and only if $j_A^\ast\eta'$ is an isomorphism, i.e., that $\eta'$ is always an isomorphism when evaluated at $a_{-1}$. This follows easily from the two diagrams
\[
\xymatrix{
(\id/a_{-1})\ar[r]^-p\ar[d]_-\pi\drtwocell\omit{}&A'\ar[r]^-=\ar[d]_-=&A'\ar[d]^-{u'}&&
(u'/u'(a_{-1}))\ar[d]\ar[r]\drtwocell\omit{}&A'\ar[d]^-{u'}\\
\bbone\ar[r]_-{a_{-1}}&A'\ar[r]_-{u'}&B',&&
\bbone\ar[r]_-{u'(a_{-1})}&B'.
}
\]
By the compatibility of mates with pasting, the canonical mate associated to the diagram on the left factors as
\[
\pi_! p^\ast\stackrel{\cong}{\to} (a_{-1})^\ast\stackrel{\eta'_{a_{-1}}}{\to} (a_{-1})^\ast (u')^\ast(u')_!
\]
of which the first map is an isomorphism by (Der4). Thus, this mate is an isomorphism if and only if $\eta'_{a_{-1}}$ is an isomorphism. Note that in both diagrams the upper left corner is a terminal category (for the case on the right recall that $A'$ is obtained from $A$ by attaching a new morphism to its target). It is then obvious that the two pastings can be identified and that the diagram on the right is homotopy exact by (Der4). Thus, $\eta'_{a_{-1}}$ is always an isomorphism as desired.

After this reduction step it suffices to consider the following factorizations
\[
\xymatrix{
A\ar[r]^-{j_A}\ar[d]_-=&A'\ar[r]^-=\ar[d]^-=&A'\ar[d]^-{u'}\ar@{}[rrd]|{=}&&
A\ar[r]^-=\ar[d]_-=&A\ar[r]^-{j_A}\ar[d]^-u\drtwocell\omit{\id}&A'\ar[d]^-{u'}\\
A\ar[r]_-{j_A}&A'\ar[r]_-{u'}&B'&&
A\ar[r]_-u&B\ar[r]_-{j_B}&B'
}
\]
of the identity transformation. Since the square to the very right is homotopy exact (statement (i) of \autoref{prop:oneexact}), we see that the mate $\beta\colon u_!j_A^\ast\to j_B^\ast u'_!$ is a natural isomorphism. Using the compatibility of mates with respect to pasting, we obtain that $j_A^\ast\eta'$ is an isomorphism if and only if $u^\ast\beta\cdot \eta j_A^\ast$ is an isomorphism if and only if $\eta j_A^\ast$ is an isomorphism. Together with the reduction step this implies that $\eta'$ is an isomorphism if and only if $\eta j_A^\ast$ is an isomorphism, concluding the proof in the first case.

The case of one-point extensions obtained by attaching a new morphism to its source is dual in that one uses the adjunction \emph{counits} to characterize the essential images of the respective restriction functors. The first reduction step is then the same while the second step is based on statement (ii) of \autoref{prop:oneexact}.

(ii) If $u$ is fully faithful, then clearly also $u'$ is fully faithful. A coherent diagram $Y'$ is in the essential image of $u'_\ast\colon \D^{A'} \to \D^{B'}$ if and only if the adjunction unit $\eta'\colon Y' \to u'_\ast u'^\ast(Y')$ is an isomorphism. By~\cite[Lemma 1.21(2)]{groth:ptstab}, this is equivalent to saying that $\eta'_{b'}\colon Y'_{b'} \to u'_\ast u'^\ast(Y')_{b'}$ is an isomorphism for each object $b' \in B'-u'(A')$. By construction, any such $b'$ can be uniquely written as $j_B(b)$ with $b \in B - u(A)$, hence $\eta'_{Y'}$ is an isomorphism if and only if $j_B^\ast \eta'_{Y'}$ is an isomorphism. To reformulate this further, we consider the diagram
\[
\xymatrix{
A\ar[r]^-u\ar[d]_-{j_A}\drtwocell\omit{\id}&B\ar[d]^-{j_B}&&A\ar[r]^-u\ar[d]_-u\drtwocell\omit{\id}&B\ar[d]^-\id\\
A'\ar[r]^-{u'}\ar[d]_-{u'}\drtwocell\omit{\id}&B'\ar[d]^-\id&=&B\ar[r]^-\id\ar[d]_-{j_B}\drtwocell\omit{\id}&
B\ar[d]^-{j_B}\\
B'\ar[r]_-\id&B'&&B'\ar[r]_-\id&B'.
}
\]
The upper-left square is homotopy exact (\autoref{prop:oneexact}) as is the lower-right square as a constant square. Since the canonical mates of the remaining two squares are the respective units, the functoriality of mates with pasting implies that $j_B^\ast \eta'_{Y'}$ is an isomorphism if and only if $\eta_{j_B^\ast Y'}$ is an isomorphism. Thus, $Y'\in\D^{B'}$ is in the essential image of $u'_\ast$ if and only if $j_B^\ast Y'$ lies in the essential image of $u_\ast.$
\end{proof}

In the following section this result will be used in the construction of the reflection functors. For instance iterated one-point extensions allow us to pass from the base case of $q\colon[2]^n\to R^n$ studied in \S\ref{sec:epibiprod} to the more complicated case in the context of arbitrary oriented trees.

\section{Abstract tilting theory for trees}
\label{sec:reflection}

In this section we construct an abstract version of reflection functors for trees and show that they induce strongly stably equivalent quivers. As an application we obtain an abstract tilting result for trees (see \autoref{thm:reflection}, its discussion, and its corollaries). We carry out the details of the strategy outlined in \S\ref{sec:reflectrep} (see also \autoref{fig:reflection}). Thus, let $Q$ be an oriented tree, let $q_0\in Q$ be a source of valence~$n$, and let $f_i\colon q_0\to q_i, 1\leq i\leq n,$ be the morphisms adjacent to~$q_0$. 

We begin with step \ref{step:one} of the strategy mentioned in \S\ref{sec:reflectrep} and hence pass to the category $Q_1$ which is obtained by gluing in an $n$-cube $[2]^n$ of length two. To make this precise let us recall the cone construction. Given a small category $A$, then the \textbf{cone} $A^\lhd$ is the small category obtained from $A$ by adjoining a new initial object $-\infty$. The cone construction is obviously functorial in $A$ and it comes with a fully faithful natural inclusion functor $i_A\colon A\to A^\lhd.$ In particular, this gives rise to the commutative diagram of small categories
\begin{equation}
\vcenter{
\xymatrix{
[1]^n_{=n-1}\ar[r]^-{w_1}\ar[d]&[1]^n_{\geq n-1}\ar[r]^-{w_2}\ar[d]&[1]^n\ar[d]\\
([1]^n_{=n-1})^\lhd\ar[r]_-{v_1}&([1]^n_{\geq n-1})^\lhd\ar[r]_-{v_2}&([1]^n)^\lhd.
}
}
\label{eq:step1preparation}
\end{equation}
(Note that these squares are not pushout squares in $\cCat$.) We observe that the morphisms $f_i\colon q_0\to q_i,1\leq i\leq n$, together define a functor $([1]^n_{=n-1})^\lhd\to Q$. This together with the construction \eqref{eq:biproductsshape} of finite biproduct diagrams via $n$-cubes gives rise to the following commutative diagram of small categories 
\begin{equation}
\vcenter{
\xymatrix{
[1]^n_{=n-1}\ar[r]^-{w_1}\ar[d]&[1]^n_{\geq n-1}\ar[r]^-{w_2}\ar[d]&[1]^n\ar[r]^-{w_3}\ar[d]&I\ar[r]^-{w_4}\ar[dd]^-{i^{(3)}}&[2]^n\ar[dd]^-{\iota_1}\\
([1]^n_{=n-1})^\lhd\ar[r]_-{v_1}\ar[d]&([1]^n_{\geq n-1})^\lhd\ar[r]_-{v_2}\ar[d]&([1]^n)^\lhd\ar[d]&&\\
Q\ar[r]_-{u_1}&Q^{(1)}\ar[r]_-{u_2}\pushoutcorner&Q^{(2)}\ar[r]_-{u_3}\pushoutcorner&
Q^{(3)}\ar[r]_-{u_4}\pushoutcorner&Q_1,\pushoutcorner
}
}
\label{eq:step1details}
\end{equation}
in which the four additional squares are pushout squares. Note that the top row is precisely \eqref{eq:biproductsshape}. Moreover, \autoref{egs:one-point-basics} immediately imply that $u_i,i=1,2$ are iterated one-point extensions of $v_i$ and similarly that $u_i,i=3,4,$ are iterated one-point extensions of $w_i$. It is easy to check that the functors $u_j, j=1,\ldots,4$, in the bottom row are fully faithful, so that also the Kan extension functors
\begin{equation}
\vcenter{
\xymatrix{
\D^Q\ar[r]^-{(u_1)_\ast}&\D^{Q^{(1)}}\ar[r]^-{(u_2)_\ast}&\D^{Q^{(2)}}\ar[r]^-{(u_3)_!}&\D^{Q^{(3)}}\ar[r]^-{(u_4)_\ast}&\D^{Q_1}
}
}
\label{eq:step1}
\end{equation}
are fully faithful. For a stable derivator~\D, recall from \autoref{prop:biproducts} the definition of the derivator $\D^{[2]^n,\mathrm{ex}}$ of coherent biproduct diagrams based on (non-invertible) $n$-cubes, and note also that the functor $\iota_1\colon[2]^n\to Q_1$ is defined in \eqref{eq:step1details}. 

\begin{prop}\label{prop:step1}
Let \D be a stable derivator. Then \eqref{eq:step1} induces a pseudo-natural equivalence between $\D^Q$ and the full subderivator $\D^{Q_1,\mathrm{ex}}$ of $\D^{Q_1}$ spanned by all diagrams~$X$ such that $\iota_1^\ast(X)$ lies in $\D^{[2]^n,\mathrm{ex}}$.
\end{prop}
\begin{proof}
Since the functors in \eqref{eq:step1} are fully faithful, their composition induces an equivalence onto the essential image. Once we know that this essential image is precisely $\D^{Q_1,\mathrm{ex}}$ this justifies that $\D^{Q_1,\mathrm{ex}}$ actually is a derivator. This equivalence is pseudo-natural with respect to exact morphisms of stable derivators because only homotopy finite Kan extensions are involved (see \cite[\S7]{ps:linearity}). We now go through the individual steps of the construction and justify that they do what they are supposed to, i.e., that the additional branches of the tree attached to the~$q_i$ do not perturb this construction. 

The first step $(u_1)_\ast$ adds only one new object to our coherent diagrams which will become the final vertex of the $n$-cube $[2]^n.$ Since $u_1\colon Q\to Q^{(1)}$ is the inclusion of a sieve, it follows that $(u_1)_\ast$ is right extension by zero, and hence induces a pseudo-natural equivalence between $\D^Q$ and the full subderivator $\D^{Q^{(1)},\mathrm{ex}}$ of $\D^{Q^{(1)}}$ spanned by all diagrams vanishing on the new object (see \autoref{lem:extbyzero}).

The second step $(u_2)_\ast$ is also fully faithful, and hence induces an equivalence onto its essential image. Since the functor $u_2$ is obtained from the fully faithful inclusion $v_2\colon ([1]^n_{\ge n-1})^\lhd \to ([1]^n)^\lhd$ in \eqref{eq:step1details} by finitely many one-point extensions, \autoref{thm:onepointessim}(ii) implies that $X$ lies in the essential image of $(u_2)_\ast$ if and only if the restriction of $X$ to $([1]^n)^\lhd$ lies in the essential image of $(v_2)_\ast$. In other words, writing $i^{(2)}\colon [1]^n\to Q^{(2)}$ for the functor in \eqref{eq:step1details}, it is easy to see that the essential image of $(u_2)_\ast$ is the full subderivator of $\D^{Q^{(2)}}$ spanned by the diagrams $X$ such that $(i^{(2)})^\ast X$ is strongly bicartesian. Let us write $\D^{Q^{(2)},\mathrm{ex}}$ for the full subderivator of $\D^{Q^{(2)}}$ of the diagrams $X$ such that $(i^{(2)})^\ast X$ is strongly bicartesian and vanishes on the final vertex. Then, this second step induces a pseudo-natural equivalence $\D^{Q^{(1)},\mathrm{ex}}\simeq\D^{Q^{(2)},\mathrm{ex}}$. 

The functor $u_3\colon Q^{(2)}\to Q^{(3)}$ is the inclusion of a cosieve and $(u_3)_!$ is hence left extension by zero. Thus, by \autoref{lem:extbyzero}, $(u_3)_!\colon\D^{Q^{(2)}}\to\D^{Q^{(3)}}$ induces an equivalence onto the full subderivator of $\D^{Q^{(3)}}$ spanned by all objects $X$ such that $(i^{(3)})^\ast(X)\in\D^I$ vanishes on the objects $(0,2,\ldots,2),\ldots,(2,\ldots,2,0)$ (the functor $i^{(3)}$ is defined via \eqref{eq:step1details}). If we write $\D^{Q^{(3)},\mathrm{ex}}$ for the full subderivator of $\D^{Q^{(3)}}$ spanned by the diagrams $X$ such that $(i^{(3)})^\ast(X)$ satisfies these vanishing conditions, makes the $n$-cube bicartesian, and also vanishes on the final object, then $(u_3)_!$ induces a pseudo-natural equivalence $\D^{Q^{(2)},\mathrm{ex}}\simeq\D^{Q^{(3)},\mathrm{ex}}$.

It remains to study the final step $(u_4)_\ast\colon \D^{Q^{(3)}}\to \D^{Q_1}$. The claim is that this functor amounts to adding strongly bicartesian cubes everywhere, and that $(u_4)_\ast$ hence induces an equivalence onto the corresponding full subderivator of $\D^{Q_1}$. The details of the proof of this claim are basically a combination of the details of the proof of \autoref{prop:biproducts} together with the observation that the additional branches of the tree attached to the $q_i$ can be ignored. To put this latter claim a bit more precisely, $u_4$ is obtained from the inclusion $w_4\colon I \to [2]^n$ in \eqref{eq:step1details} by finitely many one-point extensions, where we add $q_0$ first and then all the branches at $q_1,\dots,q_n$ of $Q$ which do not contain $q_0$. This way, we can again use \autoref{thm:onepointessim}(ii) to infer that $X \in \D^{Q_1}$ is in the essential image of $(u_4)_\ast$ if and only if the restriction of $X$ to $[2]^n$ is in the essential image of $(w_4)_\ast$, and then conclude by the proof of \autoref{prop:biproducts}.

As a summary, \eqref{eq:step1} induces pseudo-natural equivalences of stable derivators
\[
\xymatrix{
\D^Q\ar[r]^-\simeq_-{(u_1)_\ast}&\D^{Q^{(1)},\mathrm{ex}}\ar[r]^-\simeq_-{(u_2)_\ast}&\D^{Q^{(2)},\mathrm{ex}}\ar[r]^-\simeq_-{(u_3)_!}&\D^{Q^{(3)},\mathrm{ex}}\ar[r]^-\simeq_-{(u_4)_\ast}&\D^{Q_1,\mathrm{ex}},
}
\]
and it is immediate that the final derivator $\D^{Q_1,\mathrm{ex}}$ is precisely the one considered in the statement of this proposition.
\end{proof}

We now take care of steps \ref{step:two} and \ref{step:three} of the strategy outlined in \S\ref{sec:reflectrep}. This basically amounts to recycling the results of \S\ref{sec:epibiprod} and \S\ref{sec:onepoint}. More specifically, we pass from the category $Q_1$ to the category $Q_2$ obtained by `inverting the $n$-cube', i.e., we thus consider the pushout
\begin{equation}
\vcenter{
\xymatrix{
[2]^n\ar[r]^q\ar[d]_-{\iota_1}&R^n\ar[d]^-{\iota_2}\\
Q_1\ar[r]_-{u_5}&Q_2,
}
}
\label{eq:step23}
\end{equation}
where $q$ is the localization functor studied in \S\ref{sec:epibiprod} (see \autoref{cor:invhyper}). As noted prior to \autoref{prop:step1}, $\iota_1$ is an iterated one-point extension, hence \autoref{egs:one-point-basics} imply that $u_5$ is obtained from $q$ by iterated one-point extensions as well.

\begin{prop}\label{prop:step23}
The functor $u_5\colon Q_1\to Q_2$ is a homotopical epimorphism. For a derivator \D, $u_5^\ast\colon \D^{Q_2}\to \D^{Q_1}$ induces an equivalence onto the full subderivator of $\D^{Q_1}$ spanned by all diagrams $X$ such that in $\iota_1^\ast(X)\in\D^{[2]^n}$ the morphisms
\[
(i_1,\ldots,i_{k-1},0,i_{k+1},\ldots,i_n)\to (i_1,\ldots,i_{k-1},2,i_{k+1},\ldots,i_n)
\]
are sent to isomorphisms for all $i_1,\ldots,i_{k-1},i_{k+1},\ldots,i_n$ and $k$.
\end{prop}
\begin{proof}
We know that $q\colon[2]^n\to R^n$ is a homotopical epimorphism and that its essential image admits a similar description (see \autoref{cor:invhyper}). The above description of $u_5$ as an iterated one-point extension of $q$ together with \autoref{prop:onepointepi} implies that $u_5$ is also a homotopical epimorphism. Using the description of the essential image of $q^\ast$ given in \autoref{cor:invhyper}, an inductive application of \autoref{thm:onepointessim}(i) gives us the desired description of the essential image of $(u_5)^\ast$.
\end{proof}

For a stable derivator~\D, recall the definition of $\D^{Q_1,\mathrm{ex}}$ in \autoref{prop:step1} and of the derivator $\D^{R^n,\mathrm{ex}}$ of coherent biproduct diagrams based on invertible $n$-cubes as given just prior to \autoref{cor:invbiprod}. Let us denote by $\D^{Q_2,\mathrm{ex}}$ the full subderivator of $\D^{Q_2}$ spanned by the diagrams $X$ such that $\iota_2^\ast(X)$ lies in $\D^{R^n,\mathrm{ex}}$ (the functor $\iota_2\colon R^n\to Q_2$ is defined via \eqref{eq:step23}).

\begin{cor}\label{cor:step23}
Let \D be a stable derivator and let $u_5\colon Q_1\to Q_2$ be as in \eqref{eq:step23}. Then $u_5^\ast\colon\D^{Q_2}\to\D^{Q_1}$ induces a pseudo-natural equivalence $\D^{Q_2,\mathrm{ex}}\to\D^{Q_1,\mathrm{ex}}$.
\end{cor}
\begin{proof}
We saw at the end of \S\ref{sec:epibiprod} that $q^\ast\colon\D^{R^n}\to\D^{[1]^n}$ induces a pseudo-natural equivalence of stable derivators $\D^{R^n,\mathrm{ex}}\to\D^{[1]^n,\mathrm{ex}}$. It is immediate that the equivalence of \autoref{prop:step23} restricts to the desired pseudo-natural equivalence $\D^{Q_2,\mathrm{ex}}\to\D^{Q_1,\mathrm{ex}}$.
\end{proof}

We turn to the final step \ref{step:four} of the strategy described in \S\ref{sec:reflectrep}. In this step we add the cofiber of the map $q_0\to (1,\ldots,1)$ in $Q_2$, where $(1,\ldots,1)$ is the object supporting the biproduct. Recall from \eqref{eq:square} our naming convention for the objects in $\square=[1]^2$. Let $[1]\to\ulcorner$ classify the horizontal map $(0,0)\to(1,0)$ and let $\ulcorner\to\square$ be the obvious inclusion. We consider the pushout squares 
\begin{equation}
\vcenter{
\xymatrix{
[1]\ar[r]\ar[d]&\ulcorner\ar[r]\ar[d]_-{i^{(4)}}&\square\ar[d]^-{\iota_3}\\
Q_2\ar[r]_-{u_6}&Q^{(4)}\ar[r]_-{u_7}&Q_3,
}
}
\label{eq:step4details}
\end{equation}
in which the vertical map on the left classifies the arrow $q_0\to (1,\ldots,1)$ in~${Q_2}$. We denote the objects in the image of $\iota_3$ by
\[
\xymatrix{
q_0\ar[r]\ar[d]&b\ar[d]\\
z\ar[r]&q_0',
}
\]
where the name `$b$' stands for biproduct and `$z$' for zero. If \D is a derivator, then the Kan extension morphisms
\begin{equation}
\vcenter{
\xymatrix{
\D^{Q_2}\ar[r]^-{(u_6)_\ast}&\D^{Q^{(4)}}\ar[r]^-{(u_7)_!}&\D^{Q_3}
}
}
\label{eq:step4}
\end{equation}
are fully faithful as the same is true for $u_6$ and $u_7$.

\begin{prop}\label{prop:step4}
If \D is pointed, then \eqref{eq:step4} induces an equivalence between $\D^{Q_2}$ and the full subderivator of $\D^{Q_3}$ spanned by all~$X$ such that $\iota_3^\ast(X)\in\D^{\square}$ is a cofiber square, i.e., $\iota_3^\ast(X)\in\D^{\square}$ is cocartesian and vanishes on $(0,1)\in\square$. 
\end{prop}
\begin{proof}
It is enough to describe the essential images of $(u_6)_\ast$ and $(u_7)_!$. As $u_6$ is the inclusion of a sieve, it follows that $(u_6)_\ast$ is right extension by zero. Thus, \autoref{lem:extbyzero} implies that $(u_6)_\ast\colon\D^{Q_2}\to\D^{Q^{(4)}}$ induces an equivalence onto the full subderivator of $\D^{Q^{(4)}}$ spanned by all~$X$ such that ${i^{(4)}}^\ast(X)$ vanishes at $(0,1)$ (where $i^{(4)}$ is defined via \eqref{eq:step4details}).

We now turn to the functor $(u_7)_!\colon\D^{Q^{(4)}}\to\D^{Q_3}.$ Since $u_7$ is fully faithful and the complement of the image consists of precisely the object $q_0'$ only, we can detect the essential image of $(u_7)_!$ by considering the counit at~$q_0'$ (\cite[Lemma~1.21]{groth:ptstab}). Thus, $X$ lies in the essential image of $(u_7)_!$ if and only if $\epsilon_{q_0'}\colon (u_7)_!u_7^\ast(X)_{q_0'}\to X_{q_0'}$ is an isomorphism. To analyze this further we consider the pasting
\[
\xymatrix{
\ulcorner\ar[r]^-r\ar[d]&(u_7/q_0')\ar[r]\ar[d]\drtwocell\omit{}&Q^{(4)}\ar[r]^-{u_7}\ar[d]^-{u_7}&Q_3\ar[d]^-=\\
\bbone\ar[r]&\bbone\ar[r]_-{q_0'}&Q_3\ar[r]_-=&Q_3,
}
\]
where $r\colon\ulcorner\to(u_7/q_0')$ is given by
\[
\xymatrix{
q_0\ar[r]\ar[d]\ar[dr]&b\ar[d]^-{f_b}\\
z\ar[r]_-{f_z}&q_0'.
}
\]
Let us assume for the moment that $r$ is a right adjoint. Using the compatibility of mates with pasting, (Der4), and the homotopy finality of right adjoints (see \autoref{egs:htpy}) we then conclude that $\epsilon_{q_0'}$ is an isomorphism if and only if the mate of the above pasting is an isomorphism. But this pasting can be rewritten as
\[
\xymatrix{
\ulcorner\ar[r]^-\cong\ar[d]&(i_\ulcorner/(1,1))\ar[r]\ar[d]\drtwocell\omit{}&\ulcorner\ar[r]^-{i_\ulcorner}
\ar[d]_-{i_\ulcorner}&[1]^2\ar[r]^-{\iota_3}\ar[d]^-=&Q_3\ar[d]^-=\\
\bbone\ar[r]_-=&\bbone\ar[r]_-{(1,1)}&[1]^2\ar[r]_-=&[1]^2\ar[r]_-{\iota_3}&Q_3,
}
\]
where $\iota_3$ is defined via \eqref{eq:step4details}. Thus, using similar arguments, we deduce that $\epsilon_{q_0'}$ is an isomorphism if and only if the counit $(i_\ulcorner)_!(i_\ulcorner)^\ast \iota_3^\ast(X)\to \iota_3^\ast(X)$ is an isomorphism when evaluated at $(1,1)$. By a further application of \cite[Lemma~1.21]{groth:ptstab} we see that this is the case if and only if $\iota_3^\ast(X)$ is cocartesian. 

Hence, assuming that $r$ indeed is a right adjoint, we see that $(u_7)_!\colon\D^{Q^{(4)}}\to\D^{Q_3}$ induces an equivalence onto the full subderivator of $\D^{Q_3}$ of those $X$ such that $\iota_3^\ast(X)$ is cocartesian. It is easy to see that this restricts to a further equivalence between the respective full subderivators of objects vanishing at $z$. Putting these two steps together gives the desired result.

As for the existence of the left adjoint to $r\colon\ulcorner\to(u_7/q_0')$ let us consider objects of $(u_7/q_0')$ which do not lie in the image of $r$. These are pairs $(x,g\colon u_7(x)\to q_0')$ such that the structure map factors uniquely as
\begin{equation}\label{eq:unit-rl}
g = f_b \circ \eta_{(x,g)}\colon u_7(x) \to b\to q_0'.
\end{equation}
We leave it to the reader to check that the assignment which sends all such objects to $(1,0)$ and which is inverse to $r$ on the remaining objects defines a unique functor $l\colon(u_7/q_0')\to\ulcorner$. To conclude the construction of the adjunction, we specify a unit $\eta$ and a counit $\epsilon$. In fact, $\epsilon$ is chosen as the identity $lr = \id_{\ulcorner}$. The components of $\eta$ at objects in the image of $r$ are chosen to be identities as well, while the remaining ones are the maps $\eta_{(x,g)}\colon (x,g) \to (b,f_b)$ from \eqref{eq:unit-rl}. The verification of the triangular identities is left to the reader.
\end{proof}

In order to formulate the following corollary and as a preparation for the proof of \autoref{thm:reflection}, let us summarize in the following diagram part of the categories involved in the constructions so far (as spelled out in \eqref{eq:step1details}, \eqref{eq:step23}, and \eqref{eq:step4details}),
\begin{equation}
\vcenter{
\xymatrix{
&[2]^n\ar[r]\ar[d]_-{\iota_1}&R^n\ar[d]^-{\iota_2}&\square\ar[d]^-{\iota_3}\\
Q\ar[r]&Q_1\ar[r]&Q_2\ar[r]&Q_3.
}
}
\label{eq:summary}
\end{equation}
If \D is a stable derivator, then let $\D^{Q_3,\mathrm{ex}}$ be the full subderivator of $\D^{Q_3}$ spanned by all diagrams $X$ such that the restriction to $Q_2$ lies in $\D^{Q_2,\mathrm{ex}}$ and $\iota_3^\ast(X)\in\D^\square$ is a cofiber square.

\begin{cor}\label{cor:step4}
Let \D be a stable derivator. Then \eqref{eq:step4} induces a pseudo-natural equivalence $\D^{Q_2,\mathrm{ex}}\to\D^{Q_3,\mathrm{ex}}$.
\end{cor}
\begin{proof}
This follows immediately from \autoref{prop:step4} because the functors in \eqref{eq:step4} are fully faithful.
\end{proof}

The constructions carried out so far correspond to the left half of \autoref{fig:reflection}. We now indicate the changes which are necessary to implement the other half of that figure. 

\begin{con}\label{con:dual-con}
Starting with the quiver $Q'$, instead of considering \eqref{eq:step1preparation} we begin with the two upper left commutative squares in
\begin{equation}
\vcenter{
\xymatrix{
[1]^n_{=1}\ar[r]^-{w'_1}\ar[d]&[1]^n_{\leq 1}\ar[r]^-{w'_2}\ar[d]&[1]^n\ar[r]^-{w'_3}\ar[d]&I'\ar[r]^-{w'_4}\ar[dd]^-{i^{(3)}}&[2]^n\ar[dd]^-{\iota_1}\\
([1]^n_{=1})^\rhd\ar[r]_-{v'_1}\ar[d]&([1]^n_{\leq 1})^\rhd\ar[r]_-{v'_2}\ar[d]&([1]^n)^\rhd\ar[d]&&\\
Q'\ar[r]_-{u'_1}&Q'^{(1)}\ar[r]_-{u'_2}\pushoutcorner&Q'^{(2)}\ar[r]_-{u'_3}\pushoutcorner&
Q'^{(3)}\ar[r]_-{u'_4}\pushoutcorner&Q'_1.\pushoutcorner
}
}
\label{eq:dual-step1preparation}
\end{equation}
The remaining part of the above diagram is obtained by forming pushout squares in $\cCat$. This is in complete analogy to~\eqref{eq:step1details}, the only difference being that the fully faithful composition
\[
\xymatrix@1{
[1]^n\ar[r]^-{w'_3}& I'\ar[r]^-{w'_4}& [2]^n.
}
\]
is, unlike in~\eqref{eq:biproductsshape}, the inclusion $[1]^n \subseteq [2]^n\colon x\mapsto x$ and the full subcategory $I'$ is spanned from the image of this inclusion and the corners
\[ 
(2,0,0,\dots,0,0),\quad (0,2,0,\dots,0,0),\quad \dots,\quad (0,0,0,\dots,0,2). 
\]
The pushout on the left in the following diagram is completely analogous to~\eqref{eq:step23}, thereby defining the functor $u_5'\colon Q_1' \to Q_2'$, 
\[
\xymatrix{
[2]^n\ar[r]\ar[d]&R^n\ar[d]&&[1]\ar[r]\ar[d]&\lrcorner\ar[r]^-{i_\lrcorner}\ar[d]&\square\ar[d]\\
Q'_1\ar[r]_-{u'_5}&Q'_2,\pushoutcorner&&Q'_2\ar[r]_{u'_6}&(Q')^{(4)}\ar[r]_-{u_7'}\pushoutcorner&Q'_3.\pushoutcorner
}
\]
A corresponding modification of~\eqref{eq:step4} yields the two pushout squares on the right. Here, the functor $[1] \to \lrcorner$ classifying the vertical map while the functor $[1] \to Q_2'$ classifies the arrow $(1,\dots,1) \to q_0'$. 

Exactly as in the case of $Q$ one shows that, for every stable derivator \D, the above constructions induce equivalences of derivators
\begin{equation}\label{eq:dual-half}
\xymatrix@1{
\D^{Q'} \ar[r]^-\simeq & \D^{Q_1',\mathrm{ex}} \ar@{<-}[r]^-\simeq & \D^{Q_2',\mathrm{ex}} \ar[r]^-\simeq & \D^{Q_3',\mathrm{ex}}.
}
\end{equation}
Here, the derivators $\D^{Q'_k,\mathrm{ex}}\subseteq\D^{Q'_k},k=1,2,3,$ are defined by the obvious exactness properties (see \autoref{prop:step1}, \autoref{prop:step23}, and \autoref{prop:step4}).
\end{con}

The key observation now is that the categories $Q_3$ and $Q'_3$ are isomorphic and that this isomorphism induces an isomorphism of derivators $\D^{Q'_3,\mathrm{ex}} \cong \D^{Q_3,\mathrm{ex}}$. More intuitively and as indicated in \autoref{fig:reflection}, the category $Q_3$ contains both the quiver $Q$ and the reflected quiver $Q'$ as subcategories and it is completely symmetric. Since pushouts of small categories are, in general, fairly complicated we include a detailed proof of this fact. 

\begin{lem} \label{lem:symmetry-Q3}
Let \D be a stable derivator. There is an isomorphism of categories $\sigma\colon Q_3 \to Q_3'$ which induces an isomorphism of derivators $\sigma^\ast\colon \D^{Q'_3,\mathrm{ex}} \cong \D^{Q_3,\mathrm{ex}}$.
\end{lem}

\begin{proof}
Let us adopt the following notational convention. If $A$ is one of the categories obtained from $Q$ in the bottom lines of \eqref{eq:step1details}, \eqref{eq:step23} or \eqref{eq:step4details}, we denote by $A_\mathrm{red}$ the full subcategory obtained by removing object $q_0$. Similarly, if $A'$ is one of the analogous categories obtained from $Q'$ (see \autoref{con:dual-con}), then $A'_\mathrm{red}$ denotes the full subcategory obtained by removing $q_0'$. By the very definition of $Q'$ as a reflection of $Q$, we have $Q_\mathrm{red} = Q'_\mathrm{red}$ and there are pushouts of free categories
\begin{equation}
\vcenter{
\xymatrix{
[1]^n_{=n-1} \ar[r] \ar[d]_-{(q_1,\dots,q_n)} & ([1]^n_{=n-1})^\lhd \ar[d] &&
[1]^n_{=1} \ar[r] \ar[d]_-{(q_1,\dots,q_n)} & ([1]^n_{=1})^\rhd \ar[d]
\\
Q_\textrm{red} \ar[r] & Q,\pushoutcorner &&
Q'_\textrm{red} \ar[r] & Q.\pushoutcorner'
}
}
\label{eq:Qred-to-Q}
\end{equation}

We now consider the following commutative cube in $\cCat$ in which the top and the front faces appear in~\eqref{eq:step1details}, the left face is the first of the pushout squares in~\eqref{eq:Qred-to-Q}, and the back face is defined to be a pushout (the notation for the category $Q^{(2)}_\mathrm{red}$ will soon be justified),
\[
\xymatrix@R=1.5pc{
[1]^n_{=n-1} \ar[rr] \ar[dr] \ar[dd] && [1]^n \ar[dr] \ar'[d][dd] \\
& ([1]^n_{=n-1})^\lhd \ar[rr] \ar[dd] && ([1]^n)^\lhd \ar[dd] \\
Q_\textrm{red} \ar'[r][rr] \ar[dr] && Q^{(2)}_\textrm{red} \ar[dr] \\
& Q \ar[rr] && Q^{(2)}.
}
\]
%The left hand side face is the first pushout from~\eqref{eq:Qred-to-Q}, the top face is taken from~\eqref{eq:step1details}, and the front and back faces are pushouts of categories (and hence also the front face appears in~\eqref{eq:step1details}).
The universal property of the pushout square on the back induces a canonical functor $Q^{(2)}_\mathrm{red}\to Q^{(2)}$, thereby concluding the construction of the cube. Moreover, since the left, the front, and the back faces are pushouts, the composition and cancellation property of pushout squares in $\cCat$ implies that also the right face of the cube is a pushout square. 
%
%We claim that the functor $Q^{(2)}_\mathrm{red} \to Q^{(2)}$ in the cube does follow our notational convention. Indeed, since the front, back and left faces of the cube are pushouts, it follows that also the right hand side face of the cube
%
%\[
%\xymatrix{
%[1]^n \ar[r] \ar[d] & ([1]^n)^\lhd \ar[d] \\
%Q^{(2)}_\textrm{red} \ar[r] & Q^{(2)}
%}
%\] 
%
%is a pushout, by the gluing and cancellation properties of pushouts. 
This pushout exhibits $Q^{(2)}_\textrm{red} \to Q^{(2)}$ as a one-point extension (see \autoref{egs:one-point-basics}(ii) and~(iii)) and this justifies the notation for $Q^{(2)}_\mathrm{red}$. In fact, we conclude that $Q^{(2)}$ is obtained from $Q^{(2)}_\textrm{red}$ by attaching the morphism $q_0 \to (0,\dots,0)$ at its sink.

Let us now construct the following pushout squares in $\cCat$
\begin{equation}\label{eq:very-detailed}
\vcenter{
\xymatrix{
[1]^n_{=n-1} \ar[r] \ar[d] & [1]^n \ar[d] \ar[r]^-{w_4w_3} \ar[d] & [2]^n \ar[r]^-q \ar[d] & R^n \ar[d] \\
Q_\mathrm{red} \ar[r] & Q^{(2)}_\mathrm{red} \pushoutcorner \ar[r] \ar[d] & (Q_1)_\mathrm{red} \pushoutcorner \ar[r] \ar[d] & (Q_2)_\mathrm{red} \pushoutcorner \ar[d] \\
& Q^{(2)} \ar[r]_-{u_4u_3} & Q_1 \pushoutcorner \ar[r]_-{u_5} & Q_2.\pushoutcorner
}
}
\end{equation}
The leftmost square is the back face of the above cube, and the gluings of the second and third columns can be found in~\eqref{eq:step1details} and~\eqref{eq:step23}, respectively. Composing the three upper squares in \eqref{eq:very-detailed} we obtain the pushout square on the left in 
\begin{equation}
\vcenter{
\xymatrix{
[1]^n_{=n-1} \ar[r]^-g \ar[d]_-{q_1,\dots,q_n} & R^n \ar[d] &&
[1]^n_{=1} \ar[r]^-{g'} \ar[d]_-{q_1,\dots,q_n} & R^n \ar[d]
\\
Q_\mathrm{red} \ar[r] & (Q_2)_\textrm{red},\pushoutcorner &&
Q'_\mathrm{red} \ar[r] & (Q'_2)_\textrm{red}.\pushoutcorner 
}
}
\label{eq:Q2-revisited}
\end{equation}
where $g$ maps $(1,1,\dots,0,\dots,1)$ to $(2,2,\dots,1,\dots,2) \in R^n$. Moreover, if we compose the two bottom squares in \eqref{eq:very-detailed}, then we conclude by the above discussion that $Q_2$ is a one-point extension of $(Q_2)_\textrm{red}$ obtained by attaching the morphism $q_0 \to b = (1,\dots,1)$ at $b$ (\autoref{egs:one-point-basics}(iii)). 

Starting dually with $Q'$, we obtain the pushout square on the right in \eqref{eq:Q2-revisited}, in which $g'$ maps $(0,0,\dots,1,\dots,0)$ to $(0,0,\dots,1,\dots,0) \in R^n$, and $Q'_2$ is a one-point extension of $(Q'_2)_\textrm{red}$ obtained by attaching the morphism $(1,\dots,1) = b \to q'_0$ at $b$.

These two pushout squares are related by the following two isomorphisms.
\begin{enumerate}
\item Let $\phi\colon R \toiso R\colon x\mapsto 2-x$ be the automorphism of the category $R$, which corresponds to the flip with respect to the vertical axis in~\eqref{eq:category-R}, and 
\item let $\psi\colon [1]^n_{=n-1}\toiso [1]^n_{=1}$ be the isomorphism of discrete categories given by $\psi\colon (1,1,\dots,0,\dots,1) \mapsto (0,0,\dots,1,\dots,0)$.
\end{enumerate}
One now easily checks that these isomorphisms yield an isomorphism of spans
\[
\xymatrix{
Q_\mathrm{red}\ar[d]_-\id&[1]^n_{=n-1}\ar[l]_-{(q_1,\dots,q_n)}\ar[r]^-g\ar[d]_-\psi^-\cong&R^n\ar[d]^-{\phi^n}_-\cong\\
Q'_\mathrm{red}&[1]^n_{=1}\ar[r]_-{g'}\ar[l]^-{(q_1,\dots,q_n)}&R^n.
}
\]
In particular, we obtain an isomorphism of categories $\theta\colon (Q_2)_\mathrm{red} \to (Q'_2)_\mathrm{red}$ which restricts to $\id\colon Q_\mathrm{red} \to Q'_\mathrm{red}$ and which sends $b \in (Q_2)_\mathrm{red}$ to $b \in (Q'_2)_\mathrm{red}$ (compare to the third row of \autoref{fig:reflection}). 

If we glue the pushout defining the one-point extension $(Q_2)_\mathrm{red} \to Q_2$ with \eqref{eq:step4details}, we obtain the pushout square on the left in 
\[
\xymatrix{
\bbone \ar[r]^-{(1,0)} \ar[d]_-b & \square \ar[d]^-{\iota_3} &&
\bbone \ar[r]^-{(1,0)} \ar[d]_-b & \square \ar[d]^-{\iota'_3} \\
(Q_2)_\mathrm{red} \ar[r] & Q_3,\pushoutcorner &&
(Q'_2)_\mathrm{red} \ar[r] & Q'_3,\pushoutcorner
}
\]
while the square on the right is dual. Hence the isomorphism $\theta\colon (Q_2)_\mathrm{red} \toiso (Q'_2)_\mathrm{red}$ induces an isomorphism $\sigma\colon Q_3 \toiso Q'_3$ such that $\sigma\iota_3 = \iota'_3$. This isomorphism induces an isomorphism $\sigma^\ast\colon \D^{Q'_3} \cong \D^{Q_3},$ and it is straightforward to verify that it restricts to the desired equivalence $\sigma^\ast\colon \D^{Q'_3,\mathrm{ex}} \cong \D^{Q_3,\mathrm{ex}}$.
\end{proof}

We now only have to assemble the results obtained so far in order to prove our main result.

\begin{thm}\label{thm:reflection}
Let $Q$ be an oriented tree, let $q_0\in Q$ be a source, and let $Q'=\sigma_{q_0}Q$ be the reflected quiver. The quivers $Q$ and $Q'$ are strongly stably equivalent. 
\end{thm}
\begin{proof}
Let \D be a stable derivator. Then by \autoref{prop:step1}, \autoref{cor:step23}, and \autoref{cor:step4} there are pseudo-natural equivalences of stable derivators
\[
\D^Q\stackrel{\simeq}{\to}\D^{Q_1,\mathrm{ex}}\stackrel{\simeq}{\ot}\D^{Q_2,\mathrm{ex}}\stackrel{\simeq}{\to}\D^{Q_3,\mathrm{ex}}.
\]
If we instead begin with the reflected quiver $Q'$ and perform the dual constructions, we obtain a similarly defined sequence of pseudo-natural equivalences of stable derivators
\[
\D^{Q'}\stackrel{\simeq}{\to}\D^{Q'_1,\mathrm{ex}}\stackrel{\simeq}{\ot}\D^{Q'_2,\mathrm{ex}}\stackrel{\simeq}{\to}\D^{Q'_3,\mathrm{ex}};
\]
see \autoref{con:dual-con} and \eqref{eq:dual-half}. Since $(\sigma\inv)^\ast\colon \D^{Q_3,\mathrm{ex}} \cong \D^{Q'_3,\mathrm{ex}}$ by \autoref{lem:symmetry-Q3}, we obtain a sequence of pseudo-natural equivalences
\[
\D^Q\stackrel{\simeq}{\to}\D^{Q_1,\mathrm{ex}}\stackrel{\simeq}{\ot}\D^{Q_2,\mathrm{ex}}\stackrel{\simeq}{\to}\D^{Q_3,\mathrm{ex}}{\cong}
\D^{Q'_3,\mathrm{ex}}\stackrel{\simeq}{\ot}\D^{Q'_2,\mathrm{ex}}\stackrel{\simeq}{\to}\D^{Q'_1,\mathrm{ex}}\stackrel{\simeq}{\ot}\D^{Q'},
\]
which is to say that $Q$ and $Q'$ are strongly stably equivalent.
\end{proof}

\begin{defn}
Let \D be a stable derivator. The components $s_{q_0}^-\colon\D^Q\stackrel{\simeq}{\to}\D^{Q'}$ of the strong stable equivalence $s_{q_0}^-\colon Q\sse Q'$  constructed in the proof of \autoref{thm:reflection} are the \textbf{reflection functors} associated to the oriented tree~$Q$ and the source~$q_0$. The inverse strong stable equivalence $s_{q_0}^+\colon Q'\sse Q$ has components $s_{q_0}^+\colon \D^{Q'}\stackrel{\simeq}{\to}\D^Q$, which are also called reflection functors.
\end{defn}

\begin{eg}
Let $k$ be a field and let $\D_k$ be the derivator of~$k$, i.e., the homotopy derivator associated to the projective model structure on unbounded chain complexes over~$k$. Recall that there are equivalences $\D_k^Q\simeq\D_{kQ}$ and the reflection functors $(s_{q_0}^-,s_{q_0}^+)\colon Q\sse Q'$ hence specialize to an equivalence of derivators
\[
(s_{q_0}^-,s_{q_0}^+)\colon \D_{kQ}\rightleftarrows \D_{kQ'}.
\]
The underlying exact equivalence of triangulated categories $D(kQ)\simeq D(kQ')$ can be identified with the equivalence in \autoref{thm:happel} established by Happel in \cite{happel:dynkin}.
\end{eg}

\begin{rmk}
Since these reflection functors are available for \emph{every stable derivator}, we immediately get similar equivalences for the derivator $\D_R$ of a ring~$R$, for the derivator $\D_X$ of a (quasi-compact and quasi-separated) scheme~$X$, for the derivator $\D_A$ of a differential-graded algebra~$A$, for the derivator $\D_E$ of a (symmetric) ring spectrum~$E$, and for further examples arising in algebra, geometry, and topology. Moreover, these reflection functors are pseudo-natural with respect to exact morphisms of derivators. For more details on further examples of stable derivators we refer to \cite[\S5]{gst:basic}. 
\end{rmk}

Also the following corollaries have implications for all these different contexts.

\begin{cor}\label{cor:tree}
Two oriented trees are strongly stably equivalent if and only if the underlying unoriented graphs are isomorphic.
\end{cor}
\begin{proof}
It is a purely combinatorial argument that arbitrary reorientations of trees can be obtained by iterated reflections at sources and sinks; see \cite[Theorem 1.2(1)]{BGP:reflection}. Thus, different orientations of the same tree are strongly stably equivalent by an iterated use of \autoref{thm:reflection}. If on the other hand $Q,Q'$ are two oriented trees and $\D_k$ is the derivator of any field $k$, then $\D_k^Q$ and $\D_k^{Q'}$ can only be equivalent if the underlying graphs of $Q$ and $Q'$ are isomorphic by
\cite[Proposition~5.3]{gst:basic}.
\end{proof}

Let us agree that a \textbf{forest} is a disjoint union of trees.

\begin{cor}\label{cor:forest}
Let $F_1$ and $F_2$ be not necessarily connected quivers with the same forest as underlying graph. Then $F_1$ and $F_2$ are strongly stably equivalent.
\end{cor}
\begin{proof}
This is immediate from \autoref{cor:tree} and (Der1).
\end{proof}

As already mentioned in the introduction, a more detailed analysis of these reflection functors and other functors in the case of $A_n$-quivers will be in \cite{gst:Dynkin-A}. This includes the construction of universal tilting modules, certain spectral refinements of classical tilting complexes, realizing important functors in arbitrary stable homotopy theories. And in \cite{gst:acyclic} we obtain similar results for \emph{acyclic quivers}.

\appendix

\bibliographystyle{alpha}
\bibliography{tilting2}

\def\cprime{$'$}
\begin{thebibliography}{AHHK07}

\bibitem[AHHK07]{tilting}
Lidia Angeleri~H{\"u}gel, Dieter Happel, and Henning Krause, editors.
\newblock {\em Handbook of tilting theory}, volume 332 of {\em London
  Mathematical Society Lecture Note Series}.
\newblock Cambridge University Press, Cambridge, 2007.

\bibitem[APR79]{APR79}
Maurice Auslander, Mar{\'{\i}}a~In{\'e}s Platzeck, and Idun Reiten.
\newblock Coxeter functors without diagrams.
\newblock {\em Trans. Amer. Math. Soc.}, 250:1--46, 1979.

\bibitem[ARS97]{ARS97}
Maurice Auslander, Idun Reiten, and Sverre~O. Smal{\o}.
\newblock {\em Representation theory of {A}rtin algebras}, volume~36 of {\em
  Cambridge Studies in Advanced Mathematics}.
\newblock Cambridge University Press, Cambridge, 1997.
\newblock Corrected reprint of the 1995 original.

\bibitem[ASS06]{ASS:representation}
Ibrahim Assem, Daniel Simson, and Andrzej Skowro{\'n}ski.
\newblock {\em Elements of the representation theory of associative algebras.
  {V}ol. 1}, volume~65 of {\em London Mathematical Society Student Texts}.
\newblock Cambridge University Press, Cambridge, 2006.
\newblock Techniques of representation theory.

\bibitem[Ayo07a]{ayoub:1}
Joseph Ayoub.
\newblock Les six op\'erations de {G}rothendieck et le formalisme des cycles
  \'evanescents dans le monde motivique. {I}.
\newblock {\em Ast\'erisque}, (314):x+466 pp. (2008), 2007.

\bibitem[Ayo07b]{ayoub:2}
Joseph Ayoub.
\newblock Les six op\'erations de {G}rothendieck et le formalisme des cycles
  \'evanescents dans le monde motivique. {II}.
\newblock {\em Ast\'erisque}, (315):vi+364 pp. (2008), 2007.

\bibitem[BB80]{BB80}
Sheila Brenner and M.~C.~R. Butler.
\newblock Generalizations of the {B}ernstein-{G}el\cprime fand-{P}onomarev
  reflection functors.
\newblock In {\em Representation theory, {II} ({P}roc. {S}econd {I}nternat.
  {C}onf., {C}arleton {U}niv., {O}ttawa, {O}nt., 1979)}, volume 832 of {\em
  Lecture Notes in Math.}, pages 103--169. Springer, Berlin, 1980.

\bibitem[BGP73]{BGP:reflection}
I.~N. Bern{\v{s}}te{\u\i}n, I.~M. Gel{\cprime}fand, and V.~A. Ponomarev.
\newblock Coxeter functors, and {G}abriel's theorem.
\newblock {\em Uspehi Mat. Nauk}, 28(2(170)):19--33, 1973.

\bibitem[BK72]{bousfieldkan}
Aldridge~Knight Bousfield and Daniel~Marinus Kan.
\newblock {\em Homotopy limits, completions and localizations}.
\newblock Lecture Notes in Mathematics, Vol. 304. Springer-Verlag, Berlin,
  1972.

\bibitem[Bor94]{borceux1}
Francis Borceux.
\newblock {\em Handbook of categorical algebra. 1}, volume~50 of {\em
  Encyclopedia of Mathematics and its Applications}.
\newblock Cambridge University Press, Cambridge, 1994.
\newblock Basic category theory.

\bibitem[Cis03]{cisinski:idcm}
Denis-Charles Cisinski.
\newblock Images directes cohomologiques dans les cat\'egories de mod\`eles.
\newblock {\em Annales Math\'ematiques Blaise Pascal}, 10:195--244, 2003.

\bibitem[Cis06]{cisinski:presheaves}
Denis-Charles Cisinski.
\newblock Les pr\'efaisceaux comme mod\`eles type d'homotopie.
\newblock {\em Ast\'erisque}, (308), 2006.

\bibitem[Fra96]{franke:adams}
Jens Franke.
\newblock Uniqueness theorems for certain triangulated categories with an
  {A}dams spectral sequence, 1996.
\newblock Preprint.

\bibitem[Gab72]{gabriel:unzerlegbar}
Peter Gabriel.
\newblock Unzerlegbare {D}arstellungen. {I}.
\newblock {\em Manuscripta Math.}, 6:71--103; correction, ibid. 6 (1972), 309,
  1972.

\bibitem[GdlP87]{GdlP87}
P.~Gabriel and J.~A. de~la Pe{\~n}a.
\newblock Quotients of representation-finite algebras.
\newblock {\em Comm. Algebra}, 15(1-2):279--307, 1987.

\bibitem[GL91]{GL91}
Werner Geigle and Helmut Lenzing.
\newblock Perpendicular categories with applications to representations and
  sheaves.
\newblock {\em J. Algebra}, 144(2):273--343, 1991.

\bibitem[Goo92]{goodwillie:II}
Thomas~G. Goodwillie.
\newblock Calculus. {II}. {A}nalytic functors.
\newblock {\em $K$-Theory}, 5(4):295--332, 1991/92.

\bibitem[GPS14a]{gps:additivity}
Moritz Groth, Kate Ponto, and Michael Shulman.
\newblock The additivity of traces in monoidal derivators.
\newblock {\em J. K-Theory}, 14(3):422--494, 2014.

\bibitem[GPS14b]{gps:mayer}
Moritz Groth, Kate Ponto, and Michael Shulman.
\newblock {M}ayer--{V}ietoris sequences in stable derivators.
\newblock {\em Homology, Homotopy Appl.}, 16(1):265--294, 2014.

\bibitem[Gro90]{grothendieck:derivateurs}
Alexandre Grothendieck.
\newblock Les d\'erivateurs.
\newblock
  \\\url{http://people.math.jussieu.fr/~maltsin/groth/Derivateursengl.html},
  1990.

\bibitem[Gro10]{groth:scinfinity}
Moritz Groth.
\newblock Short course on $\infty$-categories.
\newblock \\\url{http://arxiv.org/abs/1007.2925}, 2010.

\bibitem[Gro13]{groth:ptstab}
Moritz Groth.
\newblock Derivators, pointed derivators, and stable derivators.
\newblock {\em Algebraic and Geometric Topology}, 13:313--374, 2013.

\bibitem[G{\v S}14a]{gst:Dynkin-A}
Moritz Groth and Jan {\v S}{\v t}ov{\'\i}{\v c}ek.
\newblock Abstract representation theory of {D}ynkin quivers of type~{A}.
\newblock arXiv:1409.5003, 2014.

\bibitem[G{\v S}14b]{gst:basic}
Moritz Groth and Jan {\v S}{\v t}ov{\'\i}{\v c}ek.
\newblock Tilting theory via stable homotopy theory.
\newblock To appear in {\em Crelle's Journal}. Available at arXiv:1401.6451,
  2014.

\bibitem[G{\v S}15]{gst:acyclic}
Moritz Groth and Jan {\v S}{\v t}ov{\'\i}{\v c}ek.
\newblock Towards abstract representation of acyclic quivers.
\newblock In preparation, 2015.

\bibitem[Hap86]{happel:dynkin}
Dieter Happel.
\newblock Dynkin algebras.
\newblock In {\em S\'eminaire d'alg\`ebre {P}aul {D}ubreil et {M}arie-{P}aule
  {M}alliavin, 37\`eme ann\'ee ({P}aris, 1985)}, volume 1220 of {\em Lecture
  Notes in Math.}, pages 1--14. Springer, Berlin, 1986.

\bibitem[Hap87]{happel:fdalgebra}
Dieter Happel.
\newblock On the derived category of a finite-dimensional algebra.
\newblock {\em Comment. Math. Helv.}, 62(3):339--389, 1987.

\bibitem[Hel88]{heller:htpy}
Alex Heller.
\newblock Homotopy theories.
\newblock {\em Memoirs of the American Mathematical Society}, 71(383):vi+78,
  1988.

\bibitem[Hov99]{hovey:modelcats}
Mark Hovey.
\newblock {\em Model Categories}, volume~63 of {\em Mathematical Surveys and
  Monographs}.
\newblock American Mathematical Society, Providence, RI, 1999.

\bibitem[HR82]{HR82}
Dieter Happel and Claus~Michael Ringel.
\newblock Tilted algebras.
\newblock {\em Trans. Amer. Math. Soc.}, 274(2):399--443, 1982.

\bibitem[Joy]{joyal:barca}
Andr\'{e} Joyal.
\newblock The {T}heory of {Q}uasi-{C}ategories and its {A}pplications.
\newblock Lectures at CRM Barcelona February 2008.

\bibitem[Lad07]{Ladkani07}
Sefi Ladkani.
\newblock Universal derived equivalences of posets.
\newblock arXiv:0705.0946, 2007.

\bibitem[Lur09]{lurie:HTT}
Jacob Lurie.
\newblock {\em Higher topos theory}.
\newblock Number 170 in Annals of Mathematics Studies. Princeton University
  Press, Princeton, NJ, 2009.

\bibitem[Mal07]{m:k-theory-deriv}
Georges Maltsiniotis.
\newblock La {$K$}-th\'eorie d'un d\'erivateur triangul\'e.
\newblock In {\em Categories in algebra, geometry and mathematical physics},
  volume 431 of {\em Contemp. Math.}, pages 341--368. Amer. Math. Soc.,
  Providence, RI, 2007.

\bibitem[Mal11]{maltsiniotis:exact}
Georges Maltsiniotis.
\newblock Carr\'es exacts homotopiques, et d\'erivateurs.
\newblock \url{http://arxiv.org/pdf/1101.4144v1}, 2011.
\newblock Preprint.

\bibitem[ML98]{maclane}
Saunders Mac~Lane.
\newblock {\em Categories For the Working Mathematician}, volume~5 of {\em
  Graduate Texts in Mathematics}.
\newblock Springer-Verlag, New York, second edition, 1998.

\bibitem[NS09]{NS09}
Pedro Nicol{\'a}s and Manuel Saor{\'{\i}}n.
\newblock Parametrizing recollement data for triangulated categories.
\newblock {\em J. Algebra}, 322(4):1220--1250, 2009.

\bibitem[Pau09]{Pauk09}
David Pauksztello.
\newblock Homological epimorphisms of differential graded algebras.
\newblock {\em Comm. Algebra}, 37(7):2337--2350, 2009.

\bibitem[PS14]{ps:linearity}
Kate Ponto and Mike Shulman.
\newblock The linearity of traces in monoidal categories and bicategories.
\newblock arXiv:1406.7854, 2014.

\bibitem[Qui67]{quillen:ha}
Daniel~Gray Quillen.
\newblock {\em Homotopical algebra}.
\newblock Lecture Notes in Mathematics, No. 43. Springer-Verlag, Berlin, 1967.

\bibitem[Sil67]{Sil67}
L.~Silver.
\newblock Noncommutative localizations and applications.
\newblock {\em J. Algebra}, 7:44--76, 1967.

\bibitem[Sto73]{Stor73}
Hans~H. Storrer.
\newblock Epimorphic extensions of non-commutative rings.
\newblock {\em Comment. Math. Helv.}, 48:72--86, 1973.

\end{thebibliography}

\end{document}